\documentclass[11pt, a4paper, reqno, english]{amsart}

\usepackage{amssymb,amscd}
\usepackage{hyperref}
\usepackage{mathrsfs}
\usepackage[all]{xy}
\usepackage{enumerate}
\RequirePackage{amsmath}
\usepackage[mathscr]{eucal}
\DeclareFontFamily{U}{wncy}{}
\DeclareFontShape{U}{wncy}{m}{n}{
<5>wncyr5
<6>wncyr6
<7>wncyr7
<8>wncyr8
<9>wncyr9
<10>wncyr10
<11>wncyr10
<12>wncyr6
<14>wncyr7
<17>wncyr8
<20>wncyr10
<25>wncyr10}{}
\DeclareMathAlphabet{\cyr}{U}{wncy}{m}{n}

\newcommand{\mtilde}{{\mathchoice
		{\widetilde{m}}
		{\widetilde{m}}
		{\rlap{$\scriptscriptstyle{m}$}\vphantom{\raise0pt\hbox{$m$}}\smash{\lower2.5pt\hbox{$\scriptscriptstyle\widetilde{\phantom{\scriptscriptstyle{m}}}$}}}
		{\rlap{$\scriptscriptstyle{m}$}\vphantom{\raise.2pt\hbox{$m$}}\smash{\lower2.05pt\hbox{$\scriptscriptstyle\widetilde{\phantom{\scriptscriptstyle{m}}}$}}}}}

\newcommand{\Mtilde}{{\mathchoice
		{\rlap{$M$}\mkern1mu\smash[b]{\lower.5pt\hbox{$\widetilde{\phantom{M}}$}}\mkern-1mu}
		{\rlap{$M$}\mkern1mu\smash[b]{\lower.5pt\hbox{$\widetilde{\phantom{M}}$}}\mkern-1mu}
		{\rlap{$\scriptstyle{M}$}\mkern1mu\smash[b]{\lower.5pt\hbox{$\widetilde{\phantom{\scriptstyle{M}}}$}}\mkern-1mu}
		{\widetilde{M}}}}

\newcommand\kbar{{\overline{k}}}

\newcommand\Xbar{{\overline{X}}}

\newcommand\Zbar{{\overline{Z}}}

\newcommand\Vbar{{\overline{V}}}

\newcommand\Mhat{{\widehat{M}}}

\newcommand\zz{\mathbf{z}}

\newcommand\A{\mathbf{A}}
\newcommand\ZZ{\mathbb{Z}}
\newcommand\QQ{\mathbb{Q}}
\newcommand\GG{\mathbb{G}}
\newcommand\GmZ{\GG_{\mathrm{m}}}

\newcommand\T{\mathcal{T}}
\newcommand{\sOint}{{\mathcal O}}
\newcommand{\sE}{{\mathscr E}}
\renewcommand{\P}{{\mathbf P}}
\newcommand{\sD}{{\mathscr D}}
\newcommand{\sX}{{\mathscr X}}

\newcommand{\sW}{{\mathscr W}}
\newcommand{\sU}{{\mathscr U}}

\newcommand{\Frob}{\mathrm{Fr}}

\newcommand{\Hom}{{\mathrm{Hom}}}

\DeclareMathOperator\Pic{Pic}
\DeclareMathOperator\Div{Div}
\DeclareMathOperator\divi{div}
\DeclareMathOperator\Gal{Gal}
\DeclareMathOperator\Br{Br}

\DeclareMathOperator\inv{inv}
\DeclareMathOperator\Cor{Cor}

\DeclareMathOperator\ord{ord}

\newtheorem{defi}{\rm{\textbf{Definition}}}[section]
\newtheorem{theorem}[defi]{Theorem}
\newtheorem{lemma}[defi]{Lemma}
\newtheorem{prop}[defi]{Proposition}
\newtheorem{cor}[defi]{Corollary}
\theoremstyle{definition}

\newtheorem*{ack}{Acknowledgements}
\newtheorem{conjecture}[defi]{Conjecture}

\newtheorem*{terminology}{Terminology}
\newtheorem{remark}[defi]{Remark}
\newtheorem*{hypHH}{Hypothesis~$(\mathrm{H}_1)$}
\numberwithin{equation}{section}

\begin{document}

\title[On the fibration method for rational points]
{On the fibration method for rational points}

\author{Dasheng Wei}

\address{Academy of Mathematics and System Science, CAS, Beijing 100190,
  P.\ R.\ China \emph{and} School of mathematical Sciences, University of  CAS, Beijing
  100049, P.R.China}

\email{dshwei@amss.ac.cn}

\date{October 14, 2019}

\begin{abstract}
We study rational points on rationally connected varieties under an assumption of strong approximation for a "simple" variety or under Schinzel's hypothesis. We also get some unconditional results. 
\end{abstract}


%
%

\maketitle


\section{Introduction}

Let~$X$ denote a smooth proper algebraic variety over
a number field~$k$.
It is customary to embed the set~$X(k)$ of rational points of~$X$
diagonally into the space of adelic points~$X(\A_k)$.
In 1970, Manin
(\cite{Ma71},\cite{Ma86}) showed that an obstruction based on the
Brauer group of varieties, now referred to as the Brauer--Manin
obstruction, can often explain failures of the Hasse principle and weak approximation.
Further work (see \cite{wit18} for a survey) has shown that for some
classes of rationally connected varieties the Brauer--Manin obstruction is the
only obstruction to the Hasse principle and weak approximation.

Assume that there exists a morphism $f : X \to \mathbb P^1_{k}$ with rationally 
connected geometrical generic fiber.
Assume that  the Brauer-Manin obstruction controls Hasse principle and 
weak approximation for rational points on  the smooth fibers. Does the 
same hold for the total space $X$? This question has been extensively 
studied.
There are two approaches to study this problem. The first approach is to study weak approximation for its universal torsors by  descent theory developed in \cite{ctsandescent2}, see for example \cite{CTSSD}, \cite{SD99}, \cite{CTSa} by  geometric methods;  see \cite{heathbrownskorobogatov}, \cite{Jo13} by the circle method; see \cite{BHB}, \cite{derenthalsmeetswei}, \cite{irving} by the sieve method; and see \cite {hsw}, \cite{bms}, \cite{browningmatthiesen} by  additive combinatorics. 

Another approach is the fibration method which can be traced back to Hasse's proof of the local-global principle of quadratic forms. Under Schinzel's hypothesis, weak approximation was first studied
by Colliot-Th\'el\`ene and Sansuc in \cite{ctsansucschinzel}. Such a conditional result (Under Schinzel's hypothesis) was largely extended in \cite{ctsd94}, \cite {ctsksd98}, \cite{wittlnm} \cite{weioneq} and so on.

 Harpaz and Wittenberg (see \cite[Theorem 9.17]{HW}) further improved the fibration method  such that most of the above mentioned results can be covered. They proposed a conjecture (Conjecture 9.1) in \cite{HW} which implies that a large class of rationally connected varieties satisfy weak approximation with the Brauer--Manin obstruction.  

\bigskip
Let $k$ be a number field, let $P_1(t),\cdots, P_n(t)$ be pairwise distinct irreducible polynomials  in $k[t]$. Let $k_i=k[t]/(P_i(t))$ and $a_i$ the class of $t$ in $k_i$, finite field extensions $L_i/k_i$  and $b_i\in k_i^*$. Let $F_i$ be the singular locus of $R_{L_i/k}(\Bbb A^1_{L_i})\setminus R_{L_i/k}(\Bbb G_{m,L_i})$, this
is a codimension $2$ closed subset of the affine space $R_{L_i/k}(\Bbb A^1_{L_i})$. Let $W$ be the closed subvariety of $\mathbb A^2_k\setminus \{(0,0)\}\times \prod_{i=1}^n (R_{L_i/k}(\mathbb A^1_{L_i})\setminus F_i)$ with coordinates $(\lambda,\mu,{\bf z}_1,\cdots,{\bf z}_n)$ defined by the system of equations
\begin{equation}\label{equ:conj}
\begin{cases}
b_1(\lambda-a_1\mu)=N_{L_1/k_1}({\bf z}_1)\\
\cdots\cdots\\
b_n(\lambda-a_n \mu)=N_{L_n/k_n}({\bf z}_{n}),
\end{cases}
\end{equation}
where $N_{L_i/k_i}$ is the norm form of the extension $L_i/k_i$.
We call such $W$ to be \emph{a variety associated to  the pairs $(P_1(t),L_1), \cdots, (P_n(t),L_n)$}. Obviously, there is a natural projection $g: W\to \mathbb P^1_k$ by sending $(\lambda,\mu,{\bf z}_1,\cdots,{\bf z}_n)$ to $[\lambda: \mu]$.

Varieties closely related to~$W$ (in fact, partial compactifications of~$W$) first appeared in \cite[Proposition 2.6.3]{ctsandescent2}, and 
in~\cite[\textsection3.3]{skorodescent} named \emph{vertical torsors}, see also~\cite[p.~391]{ctskodescent} and \cite[\textsection4.4]{skobook}.
By \cite[Proposition 1.2]{ct15}, we have $\bar k[W]^\times =\bar k^\times $ and $\Pic (\overline W)=0$, hence $\Br_1(W)=\Br_0(W)$.

The following conjecture was first mentioned in \cite[\S 9]{HW} and \cite[\S 3.4]{wit18}, and it implies Harpaz and Wittenberg's Conjecture 9.1 (see \cite[Corollary 9.10]{HW}).
\begin{conjecture} \label{conjecture:1} 
	The variety $W$ satisfies strong  approximation for integral points off any place $v_0$ of $k$.
\end{conjecture}

If $k = \Bbb Q$, Conjecture \ref{conjecture:1} holds for each of the following two cases:
\begin{enumerate}[(1)]	
	\item  $\deg(m_i) = 1$ for all $i \in \{1, . . . , n\}$, $v_0$ is a finite place of $\Bbb Q$;
	
	\item  $\sum \deg(m_i) \leq 3$ and $\deg(m_i)=1$ for at least one $i \in \{1, . . . , n\}$.
\end{enumerate}	
The case (1) is due to Matthiesen by \cite[Theorem 1.3]{matthiesen}  and its remark; it uses additive combinatorics and builds on the work of Browning and Matthiesen \cite{browningmatthiesen}. The case (2) is
a recent result of Browning and Schindler \cite{bs}, which  builds on the work of Browning and Heath-Brown \cite{BHB}.

A scheme of finite type over a field is \emph{split} if it contains a geometrically irreducible component of multiplicity $1$.   Recall here that the \emph{rank} of a fibration $f:X\to \Bbb P_k^1$ is defined to be 
the sum of degrees of  closed points of $\Bbb P_k^1$ above which the fiber of $f$ is not split.
The definition of split fibers can be found in work of  Skorobogatov \cite{skorodescent}, which is where the notion was originally  introduced to the subject.

Let~$X$ be a smooth, proper and geometrically integral variety over a number field~$k$
and let
$f:X \to \mathbb P^1_k$
be a dominant morphism with rational connected geometrical generic fiber. 
Let $U_0\subset \mathbb P_k^1$ be the dense open subset over which the fiber of $f$ is split and $\infty \in U_0$. Let $\Psi=\mathbb P_k^1\setminus U_0$. Let $P_1(t),\cdots,P_n(t)\in k[t]$ denote the irreducible monic polynomials which vanish at the points of $\Psi$. For $i\in \{1,\cdots,n\}$,  let $k_i=k[t]/(P_i(t))$,  fix $D_i$ to be a component over $k$ of multiplicity $1$ over the point corresponding to  $P_i(t)$, and let $L_{i}$  be the algebraic closure of $k_i$ in the function field $k(D_{i})$.
\smallskip

Now we state the main results of this paper (see Theorem \ref{cor:ratpointsRC} and \ref{thm:cyclic}).
\begin{theorem} \label{maintheorem}  Let~$X$ be a smooth, proper and geometrically integral variety over a number field~$k$
	and let
	$f:X \to \mathbb P^1_k$
	be a dominant morphism with rationally connected geometrical generic fiber.  Let $P_i(t) \text{ and } L_i$ be as above.
	
	Assume that: 
	
	$X_c(k)$ is dense in $X_c(\mathbf A_k)^{\Br(X_c)}$
	for every rational point~$c$ of a Hilbert subset of $\Bbb P^1_k$;
	
	 and one of the following two conditions holds:
	
	\begin{enumerate}[i)]		
	\item strong approximation off $v_0$ holds for any variety $W$ associated to the pairs $(P_1(t),L_1), \cdots, (P_n(t),L_n),$ where $v_0$ runs through  almost all finite places of $k$;
	
	\item the hypothesis $({\rm H_1})$  holds for $P_1(t),\cdots, P_n(t)$, and field extensions $L_i/k_i$ are cyclic for all $i$.  			
	\end{enumerate}
	  Then $X(k)$ is dense in~$X(\mathbf A_k)^{\Br(X)}$.
\end{theorem}

We may compare  
Theorem \ref{maintheorem} with Harpaz and Wittenberg's  fibration theorem (\cite[Theorem 9.17]{HW}). Firstly, there is a technical condition (9.9) in \cite[Theorem 9.17]{HW}, but our result does not need it. 
Secondly, 
\cite[Theorem 9.17]{HW} is requiring that their Conjecture 9.1 holds for all $L_i$ running through finite extensions $L_i/k_i$, but the $L_i$ in our result are exactly the  algebraic closure of $k_i$ in $k(D_i)$.
In fact, if the unramified Brauer group of the generic fiber is nontrivial,  then the proof of \cite[Theorem 9.17]{HW} requires that Conjecture 9.1 holds for larger fields $L_i$ than defined in Theorem \ref{maintheorem} which depend on the unramified Brauer group of the generic fiber and on all non-split fibers, such $L_i$ are hard to be estimated. Therefore, \cite[Theorem 9.17]{HW} requires that Conjecture 9.1 holds for all finite extensions $L_i/k_i$ for simplicity. In general, \cite[Conjecture 9.1]{HW} is more difficult to be established if the $L_i$ are larger.  


There are two approaches to prove \cite[Conjecture 9.1]{HW}. The first approach is by Schinzel's hypothesis under the assumption that all $L_i/k_i$ in \cite[Conjecture 9.1]{HW} are abelian (or slightly more general, \text{almost abelian}, see \cite[\S 9.2]{HW}).  There should be \emph{no} Brauer--Manin obstructions to weak approximation on smooth fibers of $f$ for \cite[Theorem 9.17]{HW} to be applicable (see \cite[Corollary 9.27]{HW}), which should be compared with the case $ii)$ in Theorem \ref{maintheorem}.   
Another approach is by  Conjecture \ref{conjecture:1}, see \cite[Corollary 9.10]{HW}; and \cite[Conjecture 9.1]{HW} can also imply a similar result with Conjecture \ref{conjecture:1} (see \cite[Corollary 5.6 ]{CWX19}).

 Combining analytic results $(1)$ and $(2)$ with Theorem \ref{maintheorem}, we may cover all unconditional results in \cite[\S 9.4]{HW} (e.g., Theorem 9.28, Corollary 9.29, Theorem 9.31)  and \cite[Theorem 1.1]{bs}, and we can get more unconditional results (see Theorem \ref{unconditional}) since the $L_i$ in Theorem \ref{maintheorem}  are exactly the algebraic closure of $k_i$ in $k(D_i)$ but that in \cite[Theorem 9.17]{HW} are required to run through all finite extensions $L_i/k_i$.

The proof of Theorem \ref{maintheorem} runs through the whole paper. In \S 2, we mainly prove the case $i)$. Our proof needs to combine descent theory (\cite{ctsandescent2}) with the fibration method in \cite{harariduke}.
For $X$ as in Theorem \ref{maintheorem}, let $(x_v)_v\in X(\A_k)^{\Br(X)}$, by a descent result (\cite{weidescent} or \cite{cdx}), there is a vertical torsor $\mathcal T$ of $X$, such that $(x_v)_v$ is the image of $(y_v)_v\in \mathcal T(\A_k)^{\Br_1(\mathcal T)}$.
  By the local description of vertical torsors, let $V$ be a fixed open subset of $X$, then $\mathcal T_V:=\mathcal  T\times_X V$ is isomorphic to $V\times_{\mathbb P^1_k} W$. By the implicit function theorem, we may assume $(y_v)_v \in (V\times_{\mathbb P^1_k} W)(\A_k)$. 
Assuming that $W$ satisfies strong approximation, there is a rational point $c'$ of  $W$ which is very close to the image of $(y_v)_v$ in $W(\A_k)$. Let $c:=g(c')\in \mathbb P^1(k)$. By a similar argument as in \cite[Theorem 9.17]{HW}, the fiber $X_c$ has a local point $(x'_v)_v$.  By the original idea of \cite{harariduke,hararifleches}, we may easily control  values of the fiber's Brauer group at $(x'_v)_v$ when $c \text{ mod } v$ is in split fibers. The most difficult thing is how to control these  values when $c \text{ mod } v$ is in non-split fibers. In \cite[Theorem 9.17]{HW}, this is done by \cite[Conjecture 9.1]{HW} and the choice of a sufficiently large  $L_i$ in this conjecture. In our case, we observe  that $X_c$ is also the fiber over $c'$  of  the natural projection $V\times_{\mathbb P^1_k} W \to W$, and that $V\times_{\mathbb P^1_k} W\subset X\times_{\mathbb P^1_k} W$ and the natural fibration $X\times_{\mathbb P^1_k} W\to W$ is split (see Lemma \ref{lemma:split}). When $c \text{ mod }v$ is in non-split fibers, in fact $x'_v$ gives a point $z'_v$ in $V\times_{\mathbb P^1_k} W$, hence in $\mathcal T_V$ over $c'$. By a similar argument as in \cite{harariduke,hararifleches}, we easily control values of the fiber's Brauer group at $z'_v$, hence values at $x'_v$. If we furthermore suppose that the fibers of $f$ satisfies weak approximation with the Brauer--Manin obstruction, then we can find a rational point of $X$ very close to $(x_v)_v$.
 
%
%
%
%

In \S 3, we mainly prove the case $ii)$ of Theorem \ref{maintheorem}. If  Brauer groups of smooth fibers are trivial, this case is just \cite[Theorem 1.1]{ctsksd98}. In general, there is the Brauer--Manin obstrutcion to weak approximation on  smooth fibers.  By the fibration method as in \cite[Theorem 1.1]{ctsksd98}, we can find a rational point $t_0$ in $\mathbb P^1_k$ such that the fiber over $t_0$ contains local points. Therefore, the key task is to control the Brauer--Manin obstruction on this fiber. If $t_0\text{ mod }v $ is in split fibers, it is classical by Harari's original idea (\cite{harariduke,hararifleches}). Therefore, the most difficult case is  when $t_0\text{ mod } v $ is in non-split fibers. Let $m_i$ be a non-split closed point of $\mathbb  P^1_k$ with residue field $k_i=k(m_i)$, and let $L_i$ be the algebraic closure of $k_i$ in $k(D_i)$, where $D_i$ is the fixed irreducible component above $m_i$ as in Theorem \ref{maintheorem}. Let $v$ be a place of $k$ with $|v|$ large enough at which $L_i/k$ is completely split, then $D_i \text{ mod }v$  splits into many components. We may choose  $t_0$ such that $t_0 \text{ mod } v\in m_i \text{ mod }v$. Since $X_{t_0} \text{ mod } v$ is just $D_i \text{ mod }v $, smooth points on these components of $ D_i \text{ mod }v $ can be lifted to smooth points of $X$ over $t_0$ at which values of the fiber's Brauer group are different and related by the Galois conjugation in some sense. Basing on this observation, for every non-split fiber, we fix some finite places of $k$ at which $L_i/k$ is completely split; we choose $t_0$ in $\mathbb P^1_k$  which is very close to $m_i$ at these places. At these fixed places, we can modify the local points coming from the different components of $D_i \text{ mod } v$ so that the Brauer--Manin obstruction vanishes on the fiber.  This idea may also be applied to study the existence of integral points (e.g., \cite[Theorem 4.4]{CWX19}).

\bigskip

We  get some unconditional results (see Corollary \ref{cor:smallrank}, \ref{cor:smallrank2} and \ref{cor:smallrank3}). 
\begin{theorem} \label{unconditional}
	Let~$X$ be a smooth, proper, geometrically integral  variety over
	a number field~$k$
	and $f:X \to \mathbb P^1_k$
	be a dominant morphism
	with rationally connected geometric generic fiber.
	Assume that one of following condition holds:
	\begin{enumerate}[(1)]	
	\item $\mathrm{rank}(f)\leq 2$;
	
	\item  $\mathrm{rank}(f)\leq 3$, and every non-split fiber $X_m$ contains a multiplicity $1$ component $D_m$ such that the algebraic closure of $k(m)$ in the function field $k(D_m)$ is quadratic over $k(m)$;
	
	\item  $k=\Bbb Q$, and $\mathrm{rank}(f)\leq 3$, and every non-split fiber $X_m$ contains a multiplicity $1$ component $D_m$ such that the algebraic closure of  $k(m)$ in the function field $k(D_m)$ is  cyclic over $k(m)$.
	\end{enumerate}	
	If~$X_c(k)$ is dense in $X_c(\mathbf A_k)^{\Br(X_c)}$ for every
	rational point~$c$ of a Hilbert subset of~$\Bbb P^1_k$, then~$X(k)$ is dense in~$X(\mathbf A_k)^{\Br(X)}$. 
\end{theorem}

Under the assumption that the fibers above a Hilbert set of rational points satisfy weak
approximation, the case (1) and (2) were also proved in \cite[Theorem 2.2]{ctsksd98}, \cite[Theorem A,B]{ctskodescent};  the case (1) had also been dealt with in \cite[Theorem 9.31]{HW} under the assumption that $k$ is totally imaginary or that the  non-split fibers of~$f$ lie over rational points of~$\Bbb P^1_k$, they also mentioned that the technical assumption is superfluous in the footnote.

\begin{terminology}
	Notation and terminology are standard. Let $k$ be a number field, $\Omega_k$ the set of all places of $k$ and  $\infty_k$ the set of all archimedean places of $k$. Write $v<\infty$ for $v\in \Omega_k\setminus \infty_k$. Let $\mathcal O_k$ be the ring of integers of $k$ and $\mathcal O_S$ the $S$-integers of $k$ for a finite set $S$ of $\Omega$ containing $\infty_k$. For each $v\in \Omega$, the completion of $k$ at $v$ is denoted by $k_v$, the completion of $\mathcal O_k$ at $v$ by $\mathcal O_v$ and the residue field at $v$ by $k(v)$ for $v<\infty_k$. Let ${\mathbf A}_k$ be the adele ring of $k$.
	
  Let $Z$ be a smooth variety over $k$. Denote  $Z(\mathbf A_k)^B$ to be the set of all $(z_v)_v \in Z(\mathbf A_k)$ satisfying $\sum_{v \in \Omega_k} \inv_v(A(z_v))
  = 0$ for each $A$ in the subgroup $B$ of the Brauer group $\Br(Z) =
  H^2_\text{\'et}(Z,\GmZ)$ of $Z$, where the map $\inv_v : \Br(k_v)
  \to \QQ/\ZZ$ is the invariant map from local class field theory.
The subgroup $\Br_0(Z)$ of constant elements in the Brauer group is
  the image of the natural map $\Br(k) \to \Br(Z)$. The algebraic
  Brauer group $\Br_1(Z)$ is the kernel of the natural map $\Br(Z) \to
  \Br(\Zbar)$, where $\Zbar = Z \times_k \kbar$.
  
 We say that
  strong approximation holds for $Z$ off a (\emph{finite or infinite}) set $S$ of places of $k$ if
  the image of the set $Z(k)$ of rational points on $Z$ is dense in the space
  $Z(\A_k^S)$ of adelic points on $Z$ outside $S$. In particular, $Z$ satisfies strong approximation off $\Omega \setminus S$ for each finite subset $S$ of $\Omega$, which is equivalent to that $Z$ satisfies weak approximation. If $S_1\subset S_2$ be two subsets of $\Omega$, then strong approximation off $S_1$ holds for $Z$ implies strong approximation off $S_2$ holds for $Z$.
\end{terminology}

\begin{ack} This work was begun while the author participated
	in the program "Reinventing rational
	points" at IHP in 2019. The hospitality
	and financial support of the institute is gratefully acknowledged. The author would like to thank Prof. J.-L. Colliot-Th\'el\`ene for his many valuable  suggestions on the original version of this paper.
    The  work is supported by National Natural Science Foundation of China (Grant Nos. 11622111 and 11631009).
\end{ack}

\section{Under Conjecture \ref{conjecture:1} and some unconditional results}

Let $k$ be a number field,  $F$ a finite field extension of $k$, and denote $\Omega_{F/k}$ to be the set of finite places of $k$ at which $F/k$ is totally split.

\begin{prop}\label{main:1}
Let $X$ be a smooth, geometrically integral  variety over $k$, endowed with a morphism $f: X\rightarrow \mathbb P_k^1$ with geometrically integral generic fiber. Assume that every fiber of $f$ contains an irreducible component of multiplicity 1 
and that $\bar k [X]^\times=\bar k^\times$. Let $U_0, \Psi, P_i(t), D_i,k_i,L_i$ and $k(D_i)$ be as in Theorem \ref{maintheorem}. Let $U'_0\subset U_0$ be an open subset, and $\infty \in U'_0$. Let $B\subset \Br(f^{-1}(U'_0))$ be a finite set.

Let $(x_v)_v\in X({\bf A}_k)$ be orthogonal to $B\cap \Br(X)+ \Br_1(X)$ with respect to the Brauer--Manin pairing.  Assume that:

 a)  strong approximation off $\Omega_{F/k}\setminus T$ holds for any variety $W$ associated to the pairs $(P_1(t),L_1), \cdots, (P_n(t),L_n)$, where  $F$ is a field extension of $k$ depending on $X$ and $B$, and $T$ runs through  all finite subsets of  $\Omega_{F/k}$.

Then there exists $c\in U'_0(k)$ and $(x'_v)_v\in X_c({\bf A}_k)$ such that $X_c$ is smooth and $(x'_v)_v$ is orthogonal to $B$ with respect to the Brauer--Manin pairing and is arbitrary close to $(x_v)_v$ in $X(\mathbf A_k)$.
\end{prop}

\begin{remark}
	 Let $F'$ be any finite extension of $k$ and $n$ any fixed positive integer. Note that $\Omega_{FF'/k}$ is an infinite subset of $\Omega_{F/k}$, then $\Omega_{FF'/k}\setminus T$ is also an infinite set. Therefore, $\Omega_{FF'/k}\setminus T$ contains infinitely many  subsets of $\Omega_{F'/k}$ of cardinality $n$.
	So, if  $W$  satisfies strong approximation  off $S$ where $S$ runs through almost all subsets of $\Omega_{F'/k}$ of cardinality $n$, then the condition ~a) in Proposition \ref{main:1} holds.   
\end{remark}

Let $V_0=U_0\setminus \infty$ and let $V=f^{-1}(V_0)\subset X$. Let $P$ be the $k$-group of multiplicative type dual to the $\Gamma_k$-module $\widehat P=\ker[\Pic(\overline X)\rightarrow \Pic(\overline V)]$. We have the exact sequence 
$$0\rightarrow \kbar[V]^\times /\kbar^\times \xrightarrow{div} \Div_{\overline X\setminus \overline V}(\overline X) \rightarrow \widehat P \rightarrow 0.$$

Let $\widehat M$ be the free $\Gamma_k$-module generated by all divisors $$\overline {D}_\infty,\overline{D}_{11},\cdots, \overline{D}_{1r_1},\cdots, \overline{D}_{n1},\cdots, \overline{D}_{nr_n},$$ 
where $\overline{D}_{ij} \subset \overline{D}_{i}$ and $1\leqslant j \leqslant r_i$.
Then $\widehat M$ is a permutation submodule of $\Div_{\overline X\setminus \overline V}(\overline X)$, its dual  $$M\cong \mathbb G_m \times R_{L_1/k_1}(\mathbb G_{m,L_1})\times\cdots\times R_{L_n/k_n}(\mathbb G_{m,L_n}).$$ We have another exact sequence
$$0\rightarrow \kbar[V_0]^\times /\kbar^\times \xrightarrow{\divi} \widehat M \rightarrow \widehat Q \rightarrow 0,$$ 
where $\widehat Q$ is the quotient of the morphism $\divi$. Obviously $\widehat Q$ is a finitely generated free $\mathbb Z$-module since all $D_i$ are of multiplicity 1, hence its dual $Q$ is a torus.  Let $\widehat R=\kbar[V_0]^\times /\kbar^\times$, $\widehat R$ is a permutation module and its dual $R$ is isomorphic to $\prod_{i=1}^{n}R_{k_i/k}(\mathbb G_{m,k_i})$. 
Then we have a commutative diagram 
\begin{equation}\label{dia:type}
\begin{CD}
0 @>>> \widehat R @>\divi{} >> \Mhat @>>> \widehat Q @>>> 0\\
@. @V \rho VV @V \eta VV @V\lambda VV \\
0 @>>> \kbar[V]^\times/\kbar^\times @>\divi{} >> \Div_{\Xbar\setminus
	\Vbar}(\Xbar) @>>> \widehat P @>>> 0,
\end{CD}
\end{equation}
where $\rho$ and $\eta$ are the natural morphism and $\lambda$ is induced by $\eta$. The morphism $\lambda$ induces an element in $\Hom_k(\widehat Q, \Pic(\Xbar))$, we also denote it by $\lambda$.

Since $\kbar [X]^\times =\kbar^\times$, by descent theory (\cite[Corollary~2.3.4]{ctsandescent2}, \cite[Proposition 4.4.1]{skobook}), the torsors of type $\lambda$ over
$X$ always exist. 


The dual of the morphism $div: \widehat R \to \Mhat $ is then given by the morphism of
$k$-tori
\begin{equation*}
d: M \to R=\prod_{i=1}^{n}R_{k_i/k}(\mathbb G_{m,k_i}),\quad \zz \mapsto (z_\infty^{-1} N_{L_1/k_1}(\zz_1),\cdots,z_\infty^{-1} N_{L_n/k_n}(\zz_n)),
\end{equation*}
where $\zz=(z_\infty, \zz_1,\cdots \zz_n)\in M$.

We have an exact sequence of tori
\begin{equation*}
1 \to Q \to M \to R \to 1.
\end{equation*}
This makes $M$ into a $R$-torsor under the torus $Q$.

We now describe the map $V \to R$ induced by the splitting $\phi: \kbar[V]^\times/\kbar^\times \rightarrow \kbar[V]^\times$. Then $$\phi\circ \rho: \widehat R=\kbar[V_0]^\times/\kbar^\times\rightarrow \kbar[V]^\times, [t-a_i]\mapsto b_i(t-a_i),$$ where $b_i\in k_i$. The induced map $V \to R$ is induced
by
\begin{equation*}
V_0 \to R=\prod_{i=1}^{n}R_{k_i/k}(\mathbb G_{m,k_i}),\quad t \mapsto (b_1(t-a_1),\cdots,b_n(t-a_n)).
\end{equation*}



By \cite[Theorem~2.3.1, Corollary~2.3.4]{ctsandescent2}, any torsor
$\T_V$ over $V$ of type $\lambda$ is the pullback of a torsor $M$ from $R$ to $V$, hence  it is isomorphic to this variety $V\times_R M$. Since 
$$V\times_R M \cong V\times_{V_0} V_0\times_R M \cong V\times_{V_0} g^{-1}(V_0) \subset X\times_{\mathbb P^1_k} W,$$  where $W$ is the variety in Conjecture \ref{conjecture:1} and  $g: W\to \Bbb P^1_k, (\lambda,\mu)\mapsto [\lambda:\mu]$ is the natural projection. Let  $p': X\times_{\mathbb P^1_k}W \to W$ be the natural projection, then $\T_V=p'^{-1}(V_0)$. Therefore,  
we have the commutative diagram \begin{equation}\label{com:basic2}
\begin{CD}
\mathcal T_V=V\times_{V_0}  g^{-1}(V_0) @>p' >> g^{-1}(V_0)\\
@V p VV @V g VV  \\
V @>f>>  V_0.
\end{CD}
\end{equation}
We extend the commutative diagram (\ref{com:basic2}) to the following commutative  diagram 
\begin{equation}\label{com:basic}
\begin{CD}
X\times_{\mathbb P^1_k}W @>p' >> W\\
@V p VV @V g VV  \\
X @>f>> \mathbb P^1_k,
\end{CD}
\end{equation}
in which we denote morphisms by same notations for simplicity as in (\ref{com:basic2}).  

\bigskip
The following lemma is in fact \cite[Theorem 3.3.3]{skorodescent}, it is not directly used in the latter part. However, the original idea is inspired by combining this lemma with \cite[Theorem 4.2.1]{harariduke}.
\begin{lemma}\label{lemma:split}
	All fibers of the morphism $p': X\times_{\mathbb P^1_k}W \to W$ are split.
\end{lemma}
\begin{proof}
	Let $y$ be a geometrical point of $W$. Suppose $g(y)\in U_0$. The fiber of $g$ over $y$ is geometrically integral, and the fiber $$p'^{-1}(y)=f^{-1}(g(y))\times_{k(g(y))} k(y)$$ by the diagram (\ref{com:basic}). Since $f^{-1}(g(y))$ is split by the definition of $U_0$, the fiber $p'^{-1}(y)$ is split. If $g(y)\not \in U_0$, then $g(y)$ is a closed point. Without loss the generality, we may assume $g(y)$ be the point in $\mathbb {P}^1_k\setminus U_0$ corresponding to the irreducible polynomial $P_1(t)$. Let $D$ denote the divisor of $W$ defined by $P_1(t)=0$, then $y\in D$. By the definition of $W$, $\overline {D}$ is a disjoint union of the prime divisors, which implies that  $L_1\subset k[D]$, hence $L_1\subset k(y)$.  By the diagram (\ref{com:basic}), the fiber $$p'^{-1}(y)=f^{-1}(g(y))\times_{k(g(y))} k(y)=f^{-1}(g(y))\times_{k_1} k(y)\supset D_1\times_{k_1} k(y).$$
	Since $L_1$ is the algebraic  closure of $k_1$ in $k(D_1)$, the fiber of $p'$ over $y$ is split.
\end{proof}

 \bigskip
 \bigskip
\begin{proof}[Proof of Proposition \ref{main:1}]

Let $(x_v)_v\in X({\bf A}_k)$ be orthogonal to $B\cap \Br(X)+ \Br_1(X)$. 
Let $V=f^{-1}(V_0)$, and $B\subset \Br(V)$. By \cite[Theorem 1.7]{weidescent} or \cite[Corollary 4.3]{cdx}, there exists a $\mathcal T$ with $p:\mathcal{T}\rightarrow X$ which is a torsor of type $\lambda$  (see (\ref{dia:type})) such that $(x_v)_v$ is the image of a point $(y_v)_v\in \mathcal T({\bf A}_k)^{\Br_1(\mathcal T)}$.  Since $\mathcal T$ is a torsor over $X$ under the torus $Q$, all fibers of $p$ are geometrically integral, hence we have 
$$p^*(B\cap \Br(X))=p_V^*(B)\cap \Br(\mathcal T),$$ 
where $p_V$ is the restriction of $p$ on $V$.
Therefore $(y_v)_v$ is orthogonal to $p_V^*(B)\cap \Br(\mathcal T)$ by the functoriality of Brauer--Manin pairing, hence 
\begin{equation}\label{Har}
(y_v)_v\in \mathcal T({\bf A}_k)^{p_V^*(B)\cap \Br(\mathcal T)+\Br_1(\mathcal T)}.
\end{equation}

Let $\Psi=\{m_1,\cdots,m_n\}$, $\Psi'=U_0\setminus U'_0:=\{m_{n+1},\cdots,m_N\}$, and then $\mathbb P^1_k\setminus U'_0=\{m_1,\cdots, m_N\}$. Let $D_i$ just be the divisor over $m_i$ in our assumption for $1\leq i \leq n$.  For  $n+1 \leq i \leq  N$, let $P_i(t)$ be the irreducible monic polynomial which vanishes at $m_i$, we may choose  $D_i$ to be a geometrically  integral divisor of multiplicity $1$ over $k_i$ since the fiber of $f$ is split at $m_i$. 

For $1\leq i \leq N$,  let $k_i=k[t]/(P_i(t))$ and $a_i$ the class of $t$ in $k_i$;  and we may choose a finite abelian extension $E_i/k(D_i)$ such that the residue of any element of $B$ at $D_i$ vanishes in  $H^1(E_i,\QQ/\ZZ)$.  Let $K_i$ be the algebraic closure of $L_i$ in $E_i$, hence $K_i/L_i$ is a finite abelian extension. For $i\in\{n+1,\cdots, N\}$, since $f$ is split at $m_i$,  we have $L_i=k_i$, hence $K_i/k_i$ is a finite abelian extension. For each $i\in  \{1,\dots, N\}$,  we denote $G_i=\Gal(E_i/k(D_i))$
and $H_i=\Gal(E_i/k(D_i)K_i) \subset G_i$.

The subvariety $W_0:=g^{-1}(V_0)$ of  $W$ satisfies the equations (\ref{equ:conj}).  Recall  $\mathcal T_V$ to be the restriction of $\mathcal T$ on $V$ and $\mathcal T_V\cong V\times _{V_0} W_0$. Then $p'$ maps $\mathcal T_V$ onto $W_0$. 
Let $$B'=\{\Cor_{L_i/k}(\chi,\mathbf z_i): 1\leqslant i \leqslant n,\chi\in H^1(K_i/L_i,\QQ/\ZZ)\},$$ 
obviously $B'\subset \Br_1(W_0)$. Write $\mathbb A_k^1=\text{Spec } k[t]:=\Bbb P^1_k\setminus \{\infty\}$. Let $$ B''=\{\Cor_{k_i/k}(\chi, t-a_i  ): n+1\leqslant i \leqslant N,\chi\in H^1(K_i/k_i,\QQ/\ZZ)\},$$
obviously $B''\subset \Br_1(U'_0)$.

Since $B$, $B'$ and $B''$ are finite groups,   $p_V^*(B)$  is a finite subgroup of $\Br(\mathcal T_V)$,   $p'^*_{W_0}(B')$ and $p^*(f^*( B''))$ are finite subsets of $\Br_1(\mathcal T_V)$. Let $A$ be the finite subgroups of $\Br(\mathcal T_V)$ generated by $p_V^*(B)$, $p'^*_{W_0}(B')$ and $p^*(f^*( B''))$.
According to (\ref{Har}), by Harari's formal lemma (see \cite[Theorem 1.4]{ctbudapest}, \cite{harariduke}),
there exists 
\begin{equation}\label{har}
(y'_v)_v \in \mathcal T_V({\bf A}_k)^{A}
\end{equation}
arbitrarily close to $(y_v)_v$ in $\mathcal T({\bf A}_k)$, hence its image $(p(y'_v))_v$ is arbitrarily close to $(x_v)_v$ in~$X({\bf A}_k)$. Furthermore, we may assume $p(y'_v)$ belongs to a smooth fiber of $f$ above $V_0$ for each $v$ by the implicit function theorem.


Let $S$ be a finite subset of $\Omega_k$ containing $\infty_k$ and the places at which we want to approximate $(x_v)_v$. We enlarge $S$ such that :

\begin{enumerate}[i)]

\item The varieties $\mathcal T$, $X$, $V$ and $W$ extends to smooth integral $\mathcal O_S $-model $\mathbf T$,  $\mathcal X$, $\mathcal V$ and $\mathcal W$, such that diagrams (\ref{com:basic2}) and (\ref{com:basic}) can be extended to the corresponding  commutative diagrams.
Let $\mathcal W_0$, $\mathcal U_0$ and $\mathcal U'_0$ be open corresponding subschemes of $\mathcal W$ and $\Bbb P^1_{\mathcal O_S}$ which extends $W_0$, $U_0$ and $U'_0$, respectively. Furthermore, we may assume that $y'_v \in \mathbf T_\mathcal V(\mathcal O_v)$ for 
all $v\not \in S$ , that $B \subset \Br(f^{-1}(\mathcal U'_0))$, that $ B' \subset  \Br(\mathcal W_0)$, and that $B''\subset \Br(\mathcal U'_0)$. 

\item At any place $v\not \in S$, $\prod_{i=1}^N P_i(t)$ is an $\mathcal O_v$-polynomial  and separable  $\text{ mod } v$, $K_i/k$ ( hence $L_i/k$) is unramified  for any $i\in \{1,\cdots,N\}$, $b_i$ is a unit for any $i\in \{1,\cdots, n\}$ (see the definition of $W$), and the order of $A$ is invertible. 

\item
For each~$i\in \{1,\cdots,N\}$, let $D_i^0$ be the  complement  of  $D_i$ by another components of $f$ over $m_i$;  let~$\sD_i$ be the reduced closed subscheme of~$D_i$ in~$\sX$, let $\sD_i^0$ be the open subscheme of $\sD_i$ associated with $D_i^0$, and let $\sE_i$ be
the normalization of~$\sD_i^0$ in $E_i/k(D_i)$.
We may choose $\sD_i^0$ to be small enough such that $\sD_i^0$ is smooth over $\mathcal O_S$ and $\sE_i$ is finite and \'etale over~$\sD_i^0$.

\item At any place $v\not \in S$, there exist smooth local points in the fiber of~$f$ over any closed point of~$\sU'_0$;  for any $1\leq i\leq N$ and any place~$w$ of~$k_i$ which is not above $S$, if~$L_i$ possesses a place of degree~$1$ above~$w$ then there exist smooth rational points in the closed fibers of $\sD_i^0 \to \mtilde_i$ over the rational point corresponding $w$, where $\mtilde_i$ is  the Zariski closure of $m_i$ in $\Bbb P^1_{\sOint_S}$. These are possible by the Lang--Weil  estimates \cite{langweil};

\item For any $1\leq i\leq N$, any place~$w$ of~$k_i$ which does not lie above $S$ and splits completely in~$L_i$ 
	and any element $\sigma$ of~$H_i$, there exists a rational point
	in the fiber of $\sD_i^0 \to \mtilde_i$ above the closed point corresponding to~$w$, such that $\sigma $ is equal to the Frobenius automorphism of the \'etale cover $\sE_i\to\sD_i^0$
	at this rational point. This is possible by the geometric Chebotarev's density theorem (see \cite[Lemma~1.2]{ekedahl}).
\end{enumerate}

Finally, we may choose $N$ distinct places $v_{1},\dots,v_N$ which are not in $S$ and each $v_i$ splits completely in~$K_i$ by Chebotarev's density theorem.
For each place $v_i$,
we fix a place $w_i$ of~$k_i$ above $v_i$
and $t_{v_i} \in k_{v_i}$
satisfying $\ord_{w_i}(t_{v_i}-a_i)=1$. For any such place $v_i$, since $\mathbf T_\mathcal  V= \mathcal V\times _{\mathcal V_0} \mathcal W$,  we may replace the integral point $y'_{v_i}$ in (\ref{har}) of $\mathcal T_V$ by an integral point of $\mathbf T_\mathcal V$ above the point of $\mathbb P^1_{\mathcal O_{v_i}}$  with coordinator $t_{v_i}$. In fact,  if $B\subset \Br_1(f^{-1}(U'_0))$, this step can be omitted.

Let $S'=S \cup\{v_{1},\dots,v_N\}$. Let $F$ be the composite  of all  fields $K_i$ with $1\leq i \leq n$.
By our assumption, $W$ satisfies strong approximation off $\Omega_{F/k}\setminus S'$. 
Therefore, there exists a point $c'=(\lambda_0,\mu_0,\mathbf{z}_0)\in W(k)$ such that $c'$ is very close to $p'(y'_v)\in W_0(k_v)$ for $v\in S'$ and $c' \in \sW(\mathcal O_v)$ for $v\not \in (\Omega_{F/k}\cup S')$. In particular, the coordinator $t_0:=u_0/v_0$ of $c:=g(c') \in \Bbb P^1(k)$ is very close to $t_{v_i}$ at $v_i$ for $i\in\{1,\dots,N\}$.

We may assume that $c'\in W_0(k)$ and that $X_c$ is smooth.
For each $v \in S$, then
we may choose a local point $x'_v \in X_c(k_v)$ arbitrarily close to $x_v$ by the implicit  function theorem.

In the following, we will construct $x'_v \in X_c(k_v)$ for $v \not \in S$.  This step is similar as in \cite[Theorem 9.17]{HW}.




For any closed point $m \in \Bbb P^1_k$, denote~$\mtilde$ is  the Zariski closure of $m$ in $\Bbb P^1_{\sOint_S}$.
For $1\leq i \leq N$,  let $\Lambda_i=\{v\not \in S: \tilde c\ \text{mod}\ v\  \in \ \mtilde_i  \ \text{mod}\ v\}.$ Obviously, all $\Lambda_i$ are finite and pairwise disjoint. For $v\in \Lambda_i$, we denote  $w$  to be the unique place of $k_i$ above $v$ corresponding to $c$, then one has $\ord_w(t_0-a_i)>0$. 
Since $t_0$ is very close to~$t_{v_i}$,
 it implies $\ord_{w_i}(t_0-a_i)=1$, hence $v_i \in \Lambda_i$.

If $v \not \in S$ and $v\not \in \bigcup_{i=1}^N\Lambda_i$, we can choose a  $k_v$-point $x'_v$ in $X_c$ by $iv)$ and Hensel's lemma.

Suppose that  $v$ belongs to some~$\Lambda_i$ and  $v\in (\Omega_{F/k}\setminus S')$, by the definition of $\Omega_{F/k}$, $K_i$ (hence $L_i$) is totally split at $v$.
Suppose that $v$ belongs to some~$\Lambda_i$ and $v\not \in (S'\cup \Omega_{F/k})$, then $\mathbf{z}_0 \mod v$ is a smooth $k(v)$-point of the equation $N_{L_i/k_i}(\mathbf z)\equiv 0 \mod v$ by the definition of $W$; hence $N_{L_i/k_i}(\mathbf z)= 0 $ has a smooth $k_v$-point by Hensel's lemma, which implies that $L_i$ contains a place of degree 1 above $v$. 
Therefore, for any $1\leq i \leq N$ and any $v \in \Lambda_i \setminus \{v_i\}$, $L_i/k$ contains a place of degree $1$ above $v$,   
hence we can choose a rational
point~$\zeta_{w,i}$ of the fiber of $\sD_i^0 \to \mtilde_i$ over~$w$ by $iv)$, 
then we may lift it to a $k_v$-point~$x'_v$ of~$X_c$ by Hensel's lemma.

Now, we only need to 
construct~$x'_v$ for remaining places $v_{1},\dots,v_N$.
For $1\leq i \leq N$ and $v \in \Lambda_i$,  we define $$n_{w,i} = \ord_w(t_0-a_i) \text { and } \sigma_i=\sum_{v \in \Lambda_i \setminus \{v_i\}} n_{w,i} \Frob_{\zeta_{w,i}} \in G_i,$$
where $\Frob_{\zeta_{w,i}}$ is the Frobenius automorphism of
the \'etale cover $\sE_i\to\sD_i^0$
at $\zeta_{w,i}$.

We claim:	
\begin{equation}\label{lem:isinhi}
	\text{for each } i \in \{1,\dots,N\}, \text{ one has } \sigma_i \in H_i.
\end{equation}

\begin{enumerate}[a)]
\item 	Suppose $i\in \{1,\cdots,n\}$. For any $v\in \Lambda_i \setminus \{v_i\}$, $x'_v \text{ mod } v$ is in a non-split fiber. In \cite[Theorem 9.17]{HW}, it is \emph{already} dealt with by the choice of sufficient large $L_i$ in their conjecture 9.1.  In our case, we observe that $x'_v$ gives a point $z'_v$ in $V\times_{\mathbb P^1_k} W$, hence in $\mathcal T_V$ over $c'$.  Replacing the fibration $X \to \P^1_k$ by the fibration $p': \mathcal T_V \to W$, we will show it by a similar argument as in \cite{harariduke,hararifleches}. 
	
	For any $\gamma=\Cor_{L_i/k}(\chi, \mathbf z_i)\in B'\subset \Br(W_0)$, since $\sum_{v \in \Omega} \inv_v \gamma(p'(y'_v))=0$ and
	since $\gamma(p'(y'_v))=0$ for $v \not \in S$, we have $$\sum_{v\in S} \inv_v \gamma(p'(y'_v))=0.$$
	Since $c'$ is arbitrarily close to~$p'(y'_v)$ for $v \in S$,
	one obtains \begin{equation} \label{vanish1}
	\sum_{v\in S} \inv_v \gamma(c')=0.
	\end{equation}
	By the global reciprocity law and (\ref{vanish1}), we have $\sum_{v\not \in S} \inv_v \gamma(c')=0.$
	Moreover, we have $\inv_v \gamma(c')=0$
	for $v \not \in  S \cup \Lambda_i$. If $v= v_i \text{ or } v\in (\Omega_{F/k}\setminus S')$,
	  by our choice, $K_i/k_i$ is totally split, hence $$\inv_{v} \gamma(c')=0 \text { and  } \chi(\Frob_w)=0,$$ 	where $\Frob_w \in \Gal(K_i/L_i)$ denotes the Frobenius automorphism at~$w$. Therefore, one obtains that $$\sum_{v\in\Lambda_i\setminus(\Omega_{F/k}\cup\{v_i\})} \inv_v \gamma(c')=0.$$
	For $v \in \Lambda_i\setminus \Omega_{F/k}$, by $ii)$, we have $t_0=\lambda_0/\mu_0\in \mathcal O_v$. Note that $t-a_i=N_{L_i/k_i}(\zz_i)$, we have $\ord_v(N_{L_i/k_i}(\zz_0))=n_{w\mkern-1mu,\mkern1mui}$. Since the fiber of $\mathcal T_V$ over $c'$ is just $X_c$ and since $x'_v$ gives the point $z'_v$ in $\mathcal T_V$ over $c'$, hence
	\begin{align*}
	\inv_v \gamma(c')=n_{w\mkern-1mu,\mkern1mui} \mkern2mu\chi(\Frob_w).
	\end{align*}
	Therefore
	\begin{align*}
	\chi(\sigma_i)=\sum_{v \in \Lambda_i\setminus(\Omega_{F/k}\cup\{v_i\})}
	n_{w\mkern-1mu,\mkern1mui} \mkern2mu\chi(\Frob_w)=\sum_{v\in\Lambda_i\setminus(\Omega_{F/k}\cup\{v_i\})} \inv_v \gamma(c')=0\rlap{\text{.}}
	\end{align*}
	Note that $\chi$ can run through $H^1(K_i/L_i,\QQ/\ZZ)$,  hence  $\sigma_i\in H_i$.

\item Suppose $i \in \{n+1,\dots,N\}$, this case is similar as in \cite[Theorem 9.17]{HW} or in \cite{harariduke,hararifleches}. For arbitrary character $\chi\in H^1(K_i/k_i,\QQ/\ZZ)$, let $\gamma=\Cor_{k_i/k}(\chi, t-a_i)\in B''$.
	Since  $\sum_{v \in \Omega} \inv_v \gamma(g(p(y'_v)))=0$ and 	since $\gamma(g(p(y'_v)))=0$ for $v \not \in  S$, one has $\sum_{v\in S} \inv_v \gamma(g(p(y'_v)))=0$. Since $(\lambda_0,\mu_0,\mathbf z_0)$ is arbitrarily close to~$p(y'_v)$ for $v \in S$,
	one obtains \begin{equation}\label{vanish2}
	\sum_{v\in S} \inv_v \gamma(t_0)=0,
	\end{equation}
	with $t_0=\lambda_0/\mu_0$.
	By the global reciprocity law and (\ref{vanish2}), we deduce 
$\sum_{v\not \in S} \inv_v \gamma(t_0)=0.$
	
	Moreover, we have $\inv_v \gamma(t_0)=0$
	for $v \not \in  S \cup \Lambda_i$. If $v= v_i$, by our choice, $K_i/k_i$ is totally split, hence $\inv_{v_i} \gamma(t_0)=0$ and $\chi(\Frob_w)=0$. Hence, \begin{equation}\label{vanish3}
	\sum_{v\in\Lambda_i\setminus\{v_i\}} \inv_v \gamma(t_0)=0.
	\end{equation}
	For $v \in \Lambda_i$,
	we have
	\begin{align*}
	\inv_v \gamma(t_0)=n_{w\mkern-1mu,\mkern1mui} \mkern2mu\chi(\Frob_w)\rlap{\text{,}}
	\end{align*}
	where $\Frob_w \in \Gal(K_i/k_i)$ denotes the Frobenius at~$w$,
	since $\gamma \in \Br(\sU'_0)$.
	By (\ref{vanish3}),
	one obtains
	\begin{align*}
	\chi(\sigma_i)=\sum_{v \in \Lambda_i\setminus \{v_i\}}
	n_{w\mkern-1mu,\mkern1mui} \mkern2mu\chi(\Frob_w)=\sum_{v\in\Lambda_i\setminus \{v_i\}} \inv_v \gamma(t_0)=0\rlap{\text{.}}
	\end{align*}
	Note that $\chi$ can run through $H^1(K_i/k_i,\QQ/\ZZ)$,  hence  $\sigma_i\in H_i$. 
\end{enumerate}

By the above claim and $v)$,  for each $v_i$ with $1\leq i \leq N$, there is
a rational point~$\zeta_{w_i,i}$ of the fiber of $\sD_i^0\to \mtilde_i$ above~$w_i$
such that $\Frob_{\zeta_{w_i,i}}=-\sigma_i$.
Therefore, 
\begin{align}
\label{eq:sumfrob}
\sum_{v \in \Lambda_i} n_{w,i} \mkern2mu\Frob_{\zeta_{w,i}}=0.
\end{align}
By Hensel's Lemma, there is  a $k_{v_i}$-point $x'_{v_i}$ of~$X_c$ which lifts $\zeta_{w_i,i}$. Hence, this adelic point $(x'_v)_{v } \in X_c(\mathbf A_k)$ is 
arbitrarily close to~$(x_v)_{v \in \Omega}$ in $X(\mathbf A_k)$, and
we only need to show that $(x'_v)_{v }\in X_c(\mathbf A_k)^B$.

It is clear that $\sum_{v \in S} \inv_v \beta(p(y'_v))=0$ for any $\beta\in B$.
For $v \in S $, since $x'_v$ is very close to~$p(y'_v)$, we have $\beta(x'_v)=\beta(p(y'_v))$.
Therefore $$\sum_{v \in S} \inv_v \beta(x'_v)=0$$ for any $\beta\in B$.
For $v \in \Omega\setminus (S \cup \Lambda_1 \cup \dots \cup \Lambda_N)$, 
it it clear that  $\beta(x'_v)=0$ for any $\beta\in B$ since $B \subseteq \Br(f^{-1}(\sU'_0))$.
So
\begin{align*}
\sum_{v \in\Omega} \inv_v \beta(x'_v) = \sum_{i=1}^N \sum_{v \in \Lambda_i} \inv_v \beta(x'_v)
\end{align*}
for any $\beta \in B$.
Since $\beta \in \Br(f^{-1}(\sU'_0))$, by \cite[Corollaire~2.4.3]{harariduke},
one has $$\inv_v\beta(x'_v)=n_{w,i} \mkern2mu\partial_{\beta,D_i}(\Frob_{\zeta_{w,i}})$$
for any $\beta \in B$,
and any $v \in \Lambda_i$ with $1\leq 1\leq N$,
 where $\partial_{\beta,D_i} \in \Hom(G_i,\QQ/\ZZ)\subset H^1(k(D_i),\QQ/\ZZ)$
is  the residue of~$\beta$ at $D_i$.
By \eqref{eq:sumfrob}, 
one obtains $\sum_{v \in \Omega}\inv_v \beta(x'_v)=0$ for any $\beta \in B$.
\end{proof}

\begin{theorem}
	\label{cor:ratpointsRC}
	Let~$X$ be a smooth, proper, geometrically integral variety over a number field~$k$
	and let
	$f:X \to \Bbb P^1_k$
	be a dominant morphism with rationally connected geometric generic fiber.   Let $P_i(t)$ and $L_i$ be as in Theorem \ref{maintheorem}.  Assume that: 
	\begin{enumerate}
	\item strong approximation off $\Omega_{F/k}\setminus T$ holds for any $W$ associated to \\ $(P_1(t),L_1), \cdots, (P_n(t),L_n)$, where $F$ is a field extension of $k$ depending on $X$ and $T$ runs through  all finite subsets of  $\Omega_{F/k}$;
	
	\item there exists a Hilbert subset $H \subset \Bbb P^1_k$ such that
	$X_c(k)$ is dense in $X_c(\mathbf A_k)^{\Br(X_c)}$
	for every rational point~$c$ of~$H$.
\end{enumerate}	
Then~$X(k)$ is dense in $X(\mathbf A_k)^{\Br(X)}$.
\end{theorem}

\begin{proof}
	Let $U_0$ be as in Theorem \ref{maintheorem}. Since the generic fiber of $f$ is rationally connected, $\Br(X_\eta)/f_\eta^*\mkern.5mu\Br(\eta)$ is finite by \cite[Lemma 8.6]{HW}.
	Choosing $U_0$ small enough, we may assume that there exists
	a finite subgroup $B \subset \Br(f^{-1}(U_0))$ such that the map $B \to \Br(X_\eta)/f_\eta^*\Br(\eta)$ is surjective.
	On the other hand, we may choose a Hilbert subset $H' \subset U_0$ such that   the map $B \to \Br(X_c)/f_c^* (\Br(k))$ is surjective at any $c\in H'(k)$ by \cite[Proposition 4.1]{HW}.
	
	By Proposition \ref{main:1}, we have	
	\begin{equation}\label{weak1}
	\bigcup_{c\in U_0(k)} X_c(\mathbf A_k)^{B} \ \ \ \text{is dense in} \ \ \ X(\mathbf A_k)^{\Br(X)}; 
	\end{equation}
	as in \cite{skofibration}, the same proof shows that $c$ can be required in rational points in $U_0\cap H\cap H'$ in (\ref{weak1}) 	 (this uses Ekedahl's version of Hilbert's irreducibility theorem \cite{ekedahl}).
	By (2), we may approximate in the fiber~$X_c$.
\end{proof}

\begin{cor}
	\label{cor:smallrank}
	Let~$X$ be a smooth, proper, geometrically integral  variety over
	a number field~$k$
	and $f:X \to \mathbb P^1_k$
	be a dominant morphism
	with rationally connected geometric generic fiber.
	Assume that
	$\mathrm{rank}(f)\leq 2$.
	If~$X_c(k)$ is dense in $X_c(\mathbf A_k)^{\Br(X_c)}$ for every
	rational point~$c$ of a Hilbert subset of~$\Bbb P^1_k$, then~$X(k)$ is dense in~$X(\mathbf A_k)^{\Br(X)}$.
\end{cor}
\begin{proof}Corollary~\ref{cor:smallrank} is due to Harari~\cite{harariduke} \cite{hararifleches} when $\mathrm{rank}(f)\leq 1$. If $\mathrm{rank}(f) = 2$, the variety $W$ (see Conjecture \ref{conjecture:1}) is isomorphic to the complement of an affine space by a close subset of codimension $2$, by  \cite{caoxu,weitorus}, $W$ satisfies strong approximation off a place $v_0$, the proof follows from Theorem \ref{cor:ratpointsRC}.
\end{proof}

\begin{remark}
Fibrations over~$\Bbb P^1_k$ with rank~$2$ had been dealt with (using the descent method)
in~\cite[Theorem~A]{ctskodescent} under the assumption that the fibers above a Hilbert set of rational points satisfy weak approximation. It had also been dealt with in \cite[Theorem 9.31]{HW} when $k$ is totally imaginary
or the non-split fibers of~$f$ lie over rational points of~$\Bbb P^1_k$. 
\end{remark}

\begin{lemma}\label{sa:quadric}
	
	 Let $Y_n$ be the punctured affine cone over  a smooth projective
quadric of dimension~$n\geq 1$, let $Z$ be a smooth open subvariety of $Y_n$ and $Y_n\setminus Z$ has codimension $\geq 2$ in $Y_n$. Then
$Z$ satisfies strong approximation with algebraic Brauer--Manin obstruction off $v_0$, where $v_0$ can run through almost all finite places of $k$. Furthermore, 
	\begin{enumerate}
	\item if $n=1$, $\Pic(\overline Z)\cong \Bbb Z/2\Bbb Z$ as abelian groups and $\Br_1(Z)/\Br_0(Z)$ is infinite; 
	\item if $n=2$, $\Pic(\overline Z)\cong \Bbb Z$ as abelian groups  and  $\Br_1(Z)/\Br_0(Z)$ is finite.
	\item if $n\geq 3$, $\Pic(\overline Z)=0$ and $\Br_1(Z)=\Br_0(Z)$, then $Y_n$ satisfies strong approximation off $v_0$.
\end{enumerate} 
\end{lemma}
\begin{proof} We may assume $Y_n\subset \Bbb A^{n+2}_k$ defined by the equation 
	\begin{equation}
	a_1x_1^2+\cdots + a_{n+2}x_{n+2}^2=0
	\end{equation}
	with all $a_i\neq 0$.
By an easy computation, we have  $\bar k[Z]^\times=\bar k^\times$ and 
$$\Pic(\overline Z)\cong\begin{cases}
\Bbb Z/2\Bbb Z\ \ \  \text{ if } n=1;\\
\Bbb Z\ \ \ \ \ \ \ \  \text{ if } n=2;\\
0 \ \ \ \ \ \ \ \ \ \text{ if } n\geq 3.
\end{cases}
$$
Then, by Hochschild-Serre's spectral sequence, one obtains 
\begin{equation}\label{br-finite}
\Br_1(Z)/\Br_0(Z)=\begin{cases}
\infty\ \ \  \ \ \ \ \ \   \text{ if } n=1;\\
\text{finite }\ \ \ \    \text{ if } n=2;\\
0 \ \ \ \ \ \ \ \ \ \       \text{ if } n\geq 3.
\end{cases}
\end{equation}

We always assume $Z(\mathbf A_k)^{\Br_1(Z)}\neq \emptyset.$
Suppose $n=1 \text{ or } 2$, then $Y_n \setminus \{(0,\cdots,0)\}$	is a smooth toric variety, hence $Z$ satisfies strong approximation with algebraic Brauer--Manin obstruction off any places $v_0$ by \cite[Theorem 1.1]{weitorus}.

Now we suppose $n\geq 3$. We will discuss it by induction on $n$. Let $Z'$ be the complement of the closed subsets $\{x_{n+1}=x_{n+2}=0\}$ and $\{x_1=\cdots =x_{n}=0\}$ in $Z$. Obviously $Y_n\setminus Z'$ also has codimension $\geq 2$ in $Y_n$. We only need to prove $Z'$ satisfies strong approximation off $v_0$ since $Z\setminus Z'$ has codimension $ 2$.  Define a morphism
$f:Z'\to \Bbb P^1_k, (x_1,\cdots, x_{n+2}) \mapsto [x_{n+1}:x_{n+2}]$. 
It is clear that any geometrical fiber of $f$ is geometric integral, and that there exists a dense open subset $U_n$ of $\mathbb P^1_k$ at which any fiber is the complement of a codimension $\geq 2$ closed subset in some $Y_{n-1}$. 
Let $v_0$ be a place such that $a_1x_1^2+a_2x_2^2+a_3x_3^2=0$ has a nontrivial resolution at $v_0$. Obviously $v_0$ can run through almost all finite places of $k$. 
If any  fiber over a rational point of $U_0$ satisfies strong approximation with algebraic Brauer--Manin obstruction off $v_0$, by (\ref{br-finite}) and an easy  argument similar as in  Proposition \ref{main:1} or as in \cite{CTH16}, $Z'$ satisfies strong approximation off $v_0$. \qedhere 

\end{proof}

\begin{cor}
	\label{cor:smallrank2}
	let~$X$ be a smooth, proper, geometrically integral  variety over
	a number field~$k$
	and $f:X \to \mathbb P^1_k$
	be a dominant morphism
	with rationally connected geometric generic fiber.
	Assume that
	$\mathrm{rank}(f)\leq 3$. Suppose every non-split fiber $X_m$ contains a multiplicity $1$ component $D_m$ such that the algebraic closure of $k(m)$ in the function field $k(D_m)$ is a quadratic field of $k(m)$.

	If~$X_c(k)$ is dense in $X_c(\mathbf A_k)^{\Br(X_c)}$ for every
	rational point~$c$ of a Hilbert subset of~$\Bbb P^1_k$, then~$X(k)$ is dense in~$X(\mathbf A_k)^{\Br(X)}$.
\end{cor}
\begin{proof} If $rank (f)\leq 2$, it follows from Corollary \ref{cor:smallrank}. Suppose $rank (f)= 3$,  the variety $W$ (see Conjecture \ref{conjecture:1}) is the punctured affine cone over the complement of a closed subset of codimension $2$ in a smooth projective quadric of dimension $4$,
 the proof follows from Theorem \ref{cor:ratpointsRC} and Lemma \ref{sa:quadric}.
\end{proof}

\begin{remark}
	Corollary \ref{cor:smallrank2} had been dealt with (using the descent method)
	in~\cite[Theorem~B]{ctskodescent} under the assumption that the fibers above a Hilbert set of rational points satisfy weak approximation. 
\end{remark}

\section{Under Schinzel's Hypothesis and an unconditional result}

Many results on  rational points are based on Schinzel's hypothesis  (see \cite{{ctsksd98}}, \cite{wittlnm} and \cite{weioneq}).
  Schinzel's hypothesis  implies~ the following hypothesis $(\mathrm{H}_1)$ by \cite[Lemma~7.1]{swdtopics}, which is similar with \cite[Hypothesis $(\mathrm {HH}_1)$]{HW}.

\begin{hypHH}
	Let~$k$ be a number field, $n$ be a positive integer and $P_1,\dots,P_n\in k[\lambda,\mu]$ be irreducible homogeneous polynomials.
	Let~$S$ be a finite subset of $\Omega_k$  containing $\infty_k$,
	large enough that
	for any $v \notin S$, there exists $(\lambda_v,\mu_v) \in \sOint_v\times \sOint_v$ such that all $P_i(\lambda_v,\mu_v) \in \mathcal O_v^\times$.
	Suppose given
	$(\lambda_v,\mu_v)\in k_v \times k_v$ with $(\lambda_v,\mu_v)\neq (0,0)$
	for $v \in S$.
	there exists $(\lambda_0,\mu_0)\in k \times k$, with $\lambda_0$ and $\mu_0$ are integral outside~$S$, such that
	\begin{enumerate}
		\item $[\lambda_0:\mu_0]$ is arbitrarily close to $[\lambda_v:\mu_v]$ in $\Bbb P^1(k_v)$ for each place $v \in S$;
		\item for each $i\in\{1,\dots,n\}$, the element $P_i(\lambda_0,\mu_0)\in k$ is a unit outside~$S$ except at one place, at which it is a uniformizer.
	\end{enumerate}
\end{hypHH}

\begin{remark}\label{HB}
	By the work of Heath-Brown and Moroz (\cite[Theorem 2]{hbm}),
	the hypothesis $(\mathrm{H}_1)$ holds when $k=\Bbb Q$, $n=1$ and $\deg(P_1)=3$, see \cite[Remark 9.7]{HW}.
\end{remark}

 Let $P_1(t), \cdots, P_n(t)\in k[t]$ be irreducible polynomials. If the hypothesis $(\rm H_1)$ holds for homogeneous polynomials associated to $P_1(t), \cdots, P_n(t)$, we say that  \emph {  $(\rm H_1)$ holds for $P_1(t), \cdots, P_n(t)$}.

%
%
%
%

\bigskip

Let~$X$ be a smooth and geometrically integral variety over a number field~$k$
and let $f:X \to \mathbb P^1_k$ be a dominant morphism.
Let  $\Br_{vert}(X):= f^*(\Br(k(\eta)))\cap \Br_1(X)$ be the vertical Brauer group of $X$,where $\eta$ is the generic point of $\mathbb P^1_k$.

\begin{prop} \label{sc-bm}  Let~$X$ be a smooth and geometrically integral variety over $k$
	and let
	$f:X \to \mathbb P^1_k$
	be a dominant morphism with geometrically integral generic fiber. Assume that every fiber of $f$ contains an irreducible component of multiplicity 1. Let $U_0, \Psi, P_i(t), D_i,k(D_i),k_i\text { and } L_i$ be as in Theorem \ref{maintheorem}. Let $U'_0\subset U_0$ be an open subset, and $\infty \in U_0$, and let $B\subset \Br(f^{-1}(U'_0))$ be a finite set.
	
	Assume that
	
	(i)  for each $i\in \{1,\cdots,n\}$, $L_i$ is cyclic over $k_i$;
	
	(ii) the hypothesis $({\rm H_1})$  holds for $P_1(t),\cdots, P_n(t)$. 
	
	Then 
	$$ \bigcup_{c\in U'_0(k)} X_c(\mathbf A_k)^{B} \ \ \ \text{is dense in} \ \ \ X(\mathbf A_k)^{B\cap \Br(X)+ \Br_{vert}(X)} $$ where $X_c$ is the fiber of $f$ over $c\in U'_0(k)$. 
\end{prop}

\begin{proof} 
     For any $(x_v)_v\in X(\mathbf A_k)^{B\cap \Br(X)+ \Br_{vert}(X)}$, since $X$ is smooth and
	geometrically integral over $k$, any $v$-adic neighborhood of
	$x_v\in X(k_v)$ is Zariski dense on $X$, by the implicit function theorem, hence we may assume  that $x_v\in V(k_v)$ is a smooth point by shrinking $x_v$, where $V=f^{-1}(U'_0)$.

Let $\Psi=\{m_1,\cdots,m_n\}$, $\Psi'=U_0\setminus U'_0:=\{m_{n+1},\cdots,m_N\}$, and then $\mathbb P^1_k\setminus U'_0=\{m_1,\cdots, m_N\}$. Let $D_i$ just be the divisor over $m_i$ in our assumption for $1\leq i \leq n$.  For  $n+1 \leq i \leq  N$, let $P_i(t)$ be the irreducible monic polynomial which vanishes at $m_i$, we may choose  $D_i$ to be a geometrically  integral divisor of multiplicity $1$ over $k_i$ since the fiber of $f$ is split at $m_i$. 

For $1\leq i \leq N$,  let $k_i=k[t]/(P_i(t))$ and $a_i$ the class of $t$ in $k_i$;  and we may choose a finite abelian extension $E_i/k(D_i)$ such that the residue of any element of $B$ at $D_i$ vanishes in  $H^1(E_i,\QQ/\ZZ)$.  Let $K_i$ be the algebraic closure of $L_i$ in $E_i$, hence $K_i/L_i$ is also a finite abelian extension. For $i\in\{1,\cdots, n\}$, replace $K_i$ be the Galois closure of $K_i/k_i$, which is also abelian over $L_i$.   For $i\in\{n+1,\cdots, N\}$, since $f$ is split at $m_i$,  we have $L_i=k_i$, hence $K_i/k_i$ is a finite abelian extension. For each $i\in  \{1,\dots, N\}$,  we denote $G_i=\Gal(E_i/k(D_i))$
and $H_i=\Gal(E_iK_i/k(D_i)K_i) \subset G_i$.

	Let $S$ be a finite subset of $\Omega_k$ containing $\infty_k$ and the places at which we want to approximate $(x_v)_v$. We enlarge $S$ such that  we can apply Hypothesis $(\rm H_1)$ for $P_1(t),\cdots, P_n(t)$ and at any place $v\not \in S$, $\prod_{i=1}^N P_i(t)$ is an $\mathcal O_v$-polynomial and separable $\text{ mod } v$, $K_i/k$ is unramified for any $i\in \{1,\cdots,N\}$; moreover, we may assume:

\begin{enumerate}[i)]

	\item 	The morphism $f: X\rightarrow \Bbb P_k^1$, the open immersions $U_0\hookrightarrow \Bbb A_k^1$, $U'_0\hookrightarrow \Bbb A_k^1$ and $f^{-1}(U'_0)\hookrightarrow X$  extends to morphisms of  their corresponding integral models $\mathcal X, \Bbb P_{\mathcal O_S}^1,\mathcal U_0$ and $ \mathcal U'_0$ over $\mathcal O_S$.  
	
	\item
	For each~$i\in \{1,\cdots,N\}$, let $D_i^0$ be the  complement  of  $D_i$ by another components of $f$ over $m_i$;  let~$\sD_i$ be the reduced closed subscheme defined by the Zariski closures of~$D_i$ in~$\sX$, let $\sD_i^0$ be the open subscheme of $\sD_i$ associated with $D_i^0$, and let $\sE_i$ be
	the normalization of~$\sD_i^0$ in $E_i/k(D_i)$.
	We choose $\sD_i^0$ to be small enough such that $\sD_i^0$ is smooth over $\mathcal O_S$ and $\sE_i$ is finite and \'etale over~$\sD_i^0$.

	\item At any place $v\not \in S$, there exist smooth local points in the fiber of~$f$ over any closed point of~$\sU'_0$.  For any $1\leq i\leq N$ and any place~$w$ of~$k_i$ which is not above $S$, if~$L_i$ possesses a place of degree~$1$ above~$w$ then there exist smooth rational points in the closed fibers of $\sD_i^0 \to \mtilde_i$ over the rational point corresponding $w$, where $\mtilde_i$ is  the Zariski closure of $m_i$ in $\Bbb P^1_{\sOint_S}$. These are possible by the Lang--Weil  estimates \cite{langweil};

	\item For any $1\leq i\leq N$, any place~$w$ of~$k_i$ which does not lie above $S$ and splits completely in~$L_i$ 
	and any element $\sigma$ of~$H_i$, there exists a rational point
	in the fiber of $\sD_i^0 \to \mtilde_i$ above the closed point corresponding to~$w$, such that $\sigma $ is equal to the Frobenius automorphism of the \'etale cover $\sE_i\to\sD_i^0$
	at this rational point. This is possible by the geometric Chebotarev's density theorem (see \cite[Lemma~1.2]{ekedahl}).
\end{enumerate}

 The following  $S_i$ play \emph{key role} in the proof, they ensure that we can control values of the fiber's Brauer group at points close to non-split fibers. 
 
 By the Chebotarev density theorem, for $i\in \{1,\cdots,n\}$, there is a finite set $T_i$ of places in $L_i$ of degree 1 over $k$ with $|\Gal(K_i/L_i)|$ distinct elements such that 
	
	(1)   for each $\sigma \in \Gal(K_i/L_i)$, there are exactly one element $\frak q\in T_i$ with sufficiently large cardinalities of the residue fields and $Frob(\frak q) =\sigma$;
	
	(2)  the set $S_i =\{ \frak q \cap k: \ \frak q \in T_i\}$ has $|\Gal(K_i/L_i)|$ distinct elements with $S\cap S_i=\emptyset$ and $S_i \cap S_j=\emptyset $ for $i\neq j$.
	
	For each $v\in S_i$ with $i\in \{1, \cdots,n\}$,
	we also fix a place~$w$ of~$k_i$ lying over~$v$
	and an element $t_v \in k_v$
	such that $\ord_w(t_v-a_i)=1$.
	
	For $i\in \{1,\cdots,n\}$,	let $v\in S_i$. There is one $\frak q\in T_i$ with $v=\frak q\cap k\in S_i$. Since $\frak q$ is of degree 1 over $k$,  it implies $(L_i)_\frak q=k_v$. Since $L_i/k_i$ is cyclic (hence Galois), $L_i$ is totally split at $\frak q$. Therefore, for the place $v$, the scheme $\mathcal D_{i}\times_{\mathcal O_S} \mathcal O_v$ splits into 
	geometrically integral components $\{ \sigma\mathcal D_{i, v}: \sigma\in \Gal(L_i/k_i)\} $, where $\mathcal D_{i, v}$ is a fixed geometrically irreducible component for $v\in S_i$ with $1\leq i\leq n$. 
	Since $\ord_w(t_v-a_i)=1$,  one obtains $t_v\in \mathcal U'_0(\mathcal O_v)$ with $\ord_v(P_i(t_v))=1$ such that there is $x''_v\in \mathcal X_{t_v}(\mathcal O_v)$ satisfying   $$ x''_v \mod v \ \ \in   \mathcal D_{i}^0(k(v)). $$
	We replace $x_v$ by $x''_v$ for each $v\in S_i, 1\leqslant i\leqslant n$.

	Recall that  $a_i$ is the image of $t$ in $k[t]/(P_i(t))=k_i$. Let \begin{align*}
	B'&=\{\Cor_{k_i/k}(t-a_i  ,\chi): 1\leqslant i \leqslant n,\chi\in H^1(K_i/k_i,\QQ/\ZZ)\},\\
	B''&=\{\Cor_{k_i/k}(t-a_i  ,\chi): n+1\leqslant i \leqslant N,\chi\in H^1(K_i/k_i,\QQ/\ZZ)\}.
	\end{align*}
Obviously $(B'\cup B'')\subset \Br_{vert}(f^{-1}(U'_0))$. Let $A$ be the finite subgroups of $\Br(X)$ generated by $B,B'$ and $B''$.   

After enlarging $S$,  we may  assume that $S\supset \bigcup_{i=1}^n S_i$ and the order of $A$ is invertible at $v\not \in S$. Thus, according to (\ref{Har}), by Harari's formal lemma (see \cite[Theorem 1.4]{ctbudapest}, \cite{harariduke}),
there exists $(x''_v)_v \in f^{-1}(U'_0)({\bf A}_k)^A$
 arbitrarily close to $(x_v)_v$ in~$X({\bf A}_k)$. Thus, each $x''_v$ belongs to a smooth fiber of $f$ for each $v$. We enlarge $S$ such that $A \subset \Br(f^{-1}(\mathcal U'_0))$ and that $\beta(x''_v)=0$ for 
all $v\not \in S$ and for all $\beta \in A$; hence \begin{equation}\label{hara}
\sum_{v\in S}\beta(x''_v)=0.
\end{equation}

	Similarly as in Proposition \ref{main:1},  
we may choose $N$ distinct places $v'_{1},\dots,v'_N\not \in S$, such that~$v'_i$ splits completely in~$K_i$ for each~$i$ by Chebotarev's density theorem.
For each place $v'_i$,
we fix a place~$w'_i$ of~$k_i$ above $v'_i$
and $t_{v'_i} \in k_{v'_i}$
satisfying $\ord_{w'_i}(t_{v'_i}-a_i)=1$. For any $v\in \{v'_{1},\dots,v'_N\}$, we may assume that $f(x''_v)=t_v$.

	For $i\in \{1,2,\cdots,n\}$, let $P_i(\lambda,\mu)$ be the corresponding homogeneous polynomial of $P_i(t)$. By the Hypothesis $(\mathrm {H_1})$, there is $(\lambda_0,\mu_0)\in \mathcal O^2_S$ such that $[\lambda_0:\mu_0]$ is very close to $f(x''_v) \in \Bbb P^1(k_v)$ for $v\in S$ and   $\ord_{v}(P_i(\lambda_0,\mu_0))=0$ for all $v\not\in S$ except one prime $v_i$ of $k$ with
	 \begin{equation}\label{equ:HH1}
       \ord_{v_i}(P_i(\lambda_0,\mu_0))=1 \text{ for } 1\leq i\leq n.
	 \end{equation}
	  We may assume that $\mu_0\neq 0$, then let $t_0:=\lambda_0/\mu_0$.
	
	We may assume that $X_{t_0}$ is smooth.
	For each $v \in S$,
	we can choose $x'_v \in X_{t_0}(k_v)$ arbitrarily close to~$x''_v$ by the implicit  function theorem.

	In the following we will construct~$x'_v$ for $v \not \in S$.
	
 For $v \not \in S$,  let $\widetilde{t_0}$ be the closed closure of the rational point $[u_0:v_0]$ in $\Bbb P^1_{\mathcal O_S}$.	
For $1\leq i \leq N$, let $\Lambda_i=\{v\not \in S: \widetilde{t}_0 \text{ mod } v \in \mtilde_i  \text{ mod } v\}$.
Obviously, all $\Lambda_i$ are finite and pairwise disjoint. For any $i\in \{1,\cdots, n\}$, by the choice of $t_0$, we have $\Lambda_i=\{v'_i\} \text{ or } \{v_i, v'_i\}$.	If $v \in \Lambda_i$, we denote  $w$  to be the unique place of $k_i$ above $v$ corresponding to $t_0$,
then one has $\ord_w(t_0-a_i)>0$.  Since $t_0$ is very close  to~$t_{v'_i}$,
we may assume that $\ord_{w'_i}(t_0-a_i)=1$, hence $v'_i \in \Lambda_i$. 	
	
\begin{lemma} \label{split} 
		Suppose $i\in \{1,\cdots, n\}$. If $\ord_{v_i}(\mu_0)>0$, then $\Lambda_i=\{v'_i\}$; if $\ord_{v_i}(\mu_0)=0$, then $\Lambda_i=\{v_i, v'_i\}$ and $K_i/k_i$ is totally split at the unique place of $k_i$ above $v_i$ corresponding to $t_0$. 
\end{lemma}
\begin{proof}
	 For any $i\in \{1,\cdots, n\}$, we Write $$P_i(t)=\prod_{j=1}^g Q_j(t)$$ where $Q_j(t)$ are irreducible polynomials over $\mathcal O_{v_i}$ for $1\leq j\leq g$.  There are $g$ primes $w_1, \cdots, w_g$ of $k_i$ above $v_i$ corresponding to irreducible polynomials $Q_1(t), \cdots, Q_g(t)$ respectively. Let $Q_j(\lambda,\mu)$ be the corresponding homogeneous polynomial of $Q_j(t)$. Since $\ord_{v_i}(P_i(\lambda_0,\mu_0))=1$ by (\ref{equ:HH1}), there is $1\leq j_0\leq g$ such that 
	\begin{equation} \label{spr} \ord_{v_i}(Q_{j_0}(\lambda_0,\mu_0))=1 \ \ \ \text{and} \ \ \ \ord_{v_i}(Q_j(\lambda_0,\mu_0))=0 \ \ \ \text{for all $j\neq j_0$.} \end{equation} 
	
	If $\ord_{v_i}(\mu_0)>0$, then $\ord_{v_i}(P_i(t_0)) \leq 0$, hence $[\lambda_0: \mu_0] \text{ mod } v\not \in \mtilde_i \text{ mod } v$. Therefore, $\Lambda_i=\{v'_i\} $.
	
	We suppose $\ord_{v_i}(\mu_0)=0$, then $\ord_{v_i}(Q_{j_0}(t_0))=1$,
	hence $Q_{j_0}(t)$ is of degree 1 and $w_{j_0}$ is a place of degree 1 over $v_i$. Therefore,  $\Lambda_i=\{v_i, v'_i\}$ and $w_{j_0}$ is the unique place of  $k_i$ above $v_i$ corresponding to $t_0$.

 For any $\Cor_{k_i/k} (\chi, t-a_i) \in B'$, we have  
	$$ 0= \sum_{v\in \Omega_k} \Cor_{k_i/k} (\chi, t_0-a_i)_v = \Cor_{k_i/k} (\chi,  t_0-a_i )_{v_i} +\sum_{v\in S} \Cor_{k_i/k} (\chi, t_0-a_i)_v $$ by the global class field theory  and (\ref{hara}),  one obtains
	$$\sum_{j=1}^g (\chi, t_0-a_i)_{w_j}=- \sum_{v\in S} \Cor_{k_i/k}(\chi, t_v-a_i)_v =0.$$ 
	By (\ref{spr}), one obtains that $$(\chi, t_0-a_i)_{w_{j_0}}=0 \ \ \text{ for all $\chi\in H^1(K_i/k_i, \Bbb Q/\Bbb Z)$. } $$  
	Since $K_i/k_i$ is abelian, one concludes that $w_{j_0}$ splits totally  in $K_i/k_i$. 
\end{proof}

	For each $v \not \in S$ and $v$ does not belong to any~$\Lambda_i$,
	by $iii)$ and  Hensel's lemma, we can choose a smooth $k_v$-point $x'_v$ in $X_{t_0}$.

	If $v\in \Lambda_i\setminus \{v'_i\}$, since $L_i$ possesses a place of degree $1$ over $v$ by Lemma \ref{split}, 
	we can choose a 
	$k(v)$-point~$\zeta_{w,i}$ of the fiber of $\sD_i^0 \to \mtilde_i$ over $t_0 \text{ mod } v$, then we may lift it to a $k_v$-point~$x'_v$ of $X_{t_0}$. Therefore, we only need to construct $x'_v$ for the remaining places $v'_1,\cdots, v'_N$.

%
%
	

	


	For $1\leq i \leq N$
	and $v \in \Lambda_i$, let $n_{w,i} = \ord_w(t_0-a_i)$, and
    let $\Frob_{\zeta_{w,i}}$ denote the Frobenius automorphism of
	the \'etale cover $\sE_i\to\sD_i^0$
	at~$\zeta_{w,i}$. For any point $z_v\in \mathcal X(\mathcal O_v)$ with $v\not \in S$, let $\bar z_v$ be the $k(v)$-point defined by $z_v \text{ mod } v$. 
	
By the following two lemmas, for $i\in\{1,\dots,n\}$ (non-split fibers), we may modify the local point $x'_{\frak p_i} $ by its Galois conjugation $z_{\frak p_i}$ at a place $\frak p_i\in S_i$, such that values of the fiber's Brauer group at points close to non-split fibers can be similarly controlled as close to split fibers.
	
	\begin{lemma}\label{lem:isinhi2} 
	\begin{enumerate}
		\item For each $i \in \{1,\dots,n\}$,  if $\Lambda_i=\{ v'_i\}$, let $\sigma_i=1_{G_i}$; if $\Lambda_i=\{v_i, v'_i\}$,
		there exists $\frak p_i\in S_i$ and  $z_{\frak p_i}\in \mathcal X'_{t_0}(\mathcal O_{\frak p_i})$ such that 
		\begin{equation}\label{sigmai}
		\sigma_i:=\Frob_{\zeta_{w,i}}- \Frob_{\bar x'_{\frak p_i}}+ \Frob_{\bar z_{\frak p_i}}\in H_i.
		\end{equation}
		\item For each $i \in \{n+1,\dots,N\}$, let	$\sigma_i=\sum_{v \in \Lambda_i \setminus \{v'_i\}} n_{w,i} \mkern2mu\Frob_{\zeta_{w,i}}$,
		we have $\sigma_i \in H_i$.
	\end{enumerate}
	\end{lemma}
    \begin{proof}  
    	Suppose $i \in \{n+1,\dots,N\}$. This proof is classic and similar as in Proposition \ref{maintheorem}.  For arbitrary character $\chi\in H^1(K_i/k_i,\QQ/\ZZ)$. Let $\gamma=\Cor_{k_i/k}(\chi, t-a_i)$, then $\gamma\in B''$.
    	Since $t_0$ is arbitrarily close to~$f(x''_v)$ for $v \in S$,
    	we have \begin{equation}\label{vanish4}
    	\sum_{v\in S} \inv_v \gamma(t_0)=0,
    	\end{equation}
    	 by (\ref{hara}).
    	By the global reciprocity law and (\ref{vanish4}), we deduce $\sum_{v\not \in S} \inv_v \gamma(t_0)=0.$

    	Moreover, we have $\inv_v \gamma(t_0)=0$
    	for $v \not \in  S \cup \Lambda_i$.  	If $v=v'_i$, by our choice, we have $K_i/k_i$ is totally split, hence $\inv_{v'_i} \gamma(t_0)=0$. 
    	 Hence, \begin{equation}\label{vanish5}
    	\sum_{v\in\Lambda_i\setminus\{v'_i\}} \inv_v \gamma(t_0)=0.
    	\end{equation}
    	For $v \in \Lambda_i$, by \cite[Corollary 2.4.3]{harariduke}, 
    	we have
    	\begin{align}
    	\inv_v \gamma(t_0)=n_{w\mkern-1mu,\mkern1mui} \mkern2mu\chi(\Frob_w)\rlap{\text{,}}
    	\end{align}
    	where $\Frob_w \in \Gal(K_i/k_i)$ denotes the Frobenius at~$w$,
    	since $\gamma \in \Br(\sU)$. 
    	By (\ref{vanish5}),
    	one obtains
    	\begin{align*}
    	\chi(\sigma_i)=\sum_{v \in \Lambda_i\setminus \{v_i\}}
    	n_{w\mkern-1mu,\mkern1mui} \mkern2mu\chi(\Frob_w)=\sum_{v\in\Lambda_i\setminus \{v_i\}} \inv_v \gamma(t_0)=0\rlap{\text{.}}
    	\end{align*}
    	Note that $\chi$ can run through $H^1(K_i/k_i,\QQ/\ZZ)$,  hence  $\sigma_i\in H_i$.

    Suppose $i \in \{1,\dots,n\}$. 	We have $\Lambda_i =\{v'_i\} \text{ or } \{v_i, v'_i\}$. Suppose $\Lambda_i =\{v'_i\}$, $(1)$ obviously holds.  
    
    Now we assume $\Lambda_i=\{v_i,v'_i\}$, hence $t_0\in \mathcal O_{v_i}$. 
    By \cite[Corollary 2.4.3]{harariduke}	and the choice of $x'_{v_i}$, we have $$ \gamma(x'_{v_i}) =\partial_{\gamma,D_i} (\Frob_{\zeta_{w,i}}) $$
    for all $\gamma=\Cor_{k_i/k}(\chi, t-a_i)\in B'$.
    By global class field theory, we
    $$ \sum_{v\in \Omega_k} \gamma(x'_{v_i}) = \sum_{v\in \Omega_k} \Cor_{k_i/k}(\chi, t_0-a_i)_v=0  .$$  One obtains $$\gamma(x'_{v_i})=-\sum_{v\in S}\gamma(x'_{v_i})=0,$$ hence
    $$\chi (\Frob_{\zeta_{w,i}})=0 \ \ \ \text{for all} \ \chi \in  H^1(K_i/k_i, \Bbb Q/\Bbb Z).$$  
    For any $\sigma\in G_i$, denote $\overline \sigma$ to be its image in $\Gal(k(D_i)K_i/k(D_i))\cong \Gal(K_i/L_i)$.
    This implies that $\overline \Frob_{\zeta_{w,i}}\in [\Gal(K_i/k_i), \Gal(K_i/k_i)]$, the commutator of $\Gal(K_i/k_i)$. Denote  
    $$ [\Gal(L_i/k_i), \Gal(K_i/L_i)]:= \{ \sum_{\rho \in \Gal(L_i/k_i)} (\rho \circ \sigma_{\rho}  - \sigma_{\rho}): \sigma_\rho \in \Gal(K_i/L_i) \}, $$
    where $\rho \circ \sigma_{\rho}$ is the conjugate action of $\sigma_{\rho}$ by $\rho$.
    Then $$[\Gal(L_i/k_i), \Gal(K_i/L_i)] \subseteq  [\Gal(K_i/k_i), \Gal(K_i/k_i)]\subseteq \Gal(K_i/L_i) .$$ 
    Since $\Gal(L_i/k_i)$ is cyclic, one has $$H^2(L_i/k_i, \Bbb Q/\Bbb Z)=H^3(L_i/k_i, \Bbb Z)=H^1(L_i/k_i, \Bbb Z)=0 .$$
    Therefore, the restriction map 
    $$ H^1(K_i/k_i, \Bbb Q/\Bbb Z) \rightarrow H^1(K_i/L_i, \Bbb Q/\Bbb Z)^{\Gal(L_i/k_i)}  $$ is surjective. Hence the natural map
    $$ \Gal(K_i/L_i)/ [\Gal(L_i/k_i), \Gal(K_i/L_i)]  \rightarrow \Gal(K_i/k_i) /[\Gal(K_i/k_i), \Gal(K_i/k_i)] $$ is injective by the duality. One concludes that $$[\Gal(L_i/k_i), \Gal(K_i/L_i)] = [\Gal(K_i/k_i), \Gal(K_i/k_i)].$$ 
Since $\Gal(L_i/k_i)$ is cyclic, we have $$[\Gal(L_i/k_i), \Gal(K_i/L_i)]=\{\rho_i \circ \tau -\tau : \tau \in \Gal(K_i/L_i)\}$$ by an easy computation, where $\rho_i$ is a generator of $\Gal(L_i/k_i)$. 
    Then 
    \begin{equation} \label{kill}  \overline \Frob_{\zeta_{w,i}}= \rho_i \circ \tau_i- \tau_i  \ \ \text{for some} \ \ \tau_i\in \Gal(K_i/L_i) . \end{equation} 
    
    For $\tau_i$ in (\ref{kill}), there is $\frak q_i\in T_i$ such that $$-\tau_i = Frob(\frak q_i) \ \ \ \text{ and } \ \ \ (L_i)_{\frak q_i} = k_{\frak p_i}  \ \text{ with } \ \frak p_i=\frak q_i\cap k. $$

   For $v=\frak p_i$, define 
    $ z_{\frak p_i} = \rho_i( x'_{\frak p_i}) \in X_{t_0}(k_{v_i})$. Then $\Frob_{\bar z_{\frak p_i}} =-\rho_i \circ \tau_i.$
    
    We have 
    $$\overline \sigma_i=\overline\Frob_{\zeta_{w,i}}- \overline \Frob_{\bar x'_{\frak p_i}}+ \overline \Frob_{\bar z_{\frak p_i}} =\rho_i \circ \tau_i -\tau_i+ \tau_i - \rho_i \circ \tau_i=0_{\Gal(K_i/L_i)},$$
    hence $\sigma_i\in H_i$. \qedhere

    \end{proof}

    One has  
\begin{equation} \label{twist-value} \beta(z_{\frak p_i})=\beta(\gamma_i( x'_{\frak p_i}))=  -\partial_{\beta,D_i} (\gamma_i \circ \tau_i) \end{equation}
and 
\begin{equation} \label{br-value}  \beta(x'_{\frak p_i})=- \partial_{\beta,D_i} (\tau_i ) \end{equation}
for all $\beta\in B$.

Let $I=\{i:1\leq i \leq n \text{ and } \Lambda_i=\{v_i, v'_i\}\} $ and let $S'=\{\frak p_i: i\in I\}\subset S$. 
\begin{lemma}\label{br-tr} 
For any $\beta\in B$, we have $$\sum_{v\not\in  S'\cup \{ v'_1,\cdots, v'_N\}}\beta(x'_v)+ \sum_{v\in  S'}\beta(z_v) =\sum_{i=1}^{N}\partial_{\beta,D_i}(\sigma_i) .$$
\end{lemma}
\begin{proof} We have $$\sum_{v\in S\setminus S'}\beta(x'_v)=\sum_{v\in S}\beta(x'_v)-\sum_{v\in S'}\beta(x'_v)=-\sum_{v\in S'}\beta(x'_v).$$ 
Since $\beta(x'_v)=0$ for any $v\not \in \bigcup_i {\Lambda_i}\bigcup S$, one obtains
	\begin{align*}
	&\sum_{v\not\in  S'\cup \{ v'_1,\cdots, v'_N\}}\beta(x'_v)+ \sum_{v\in  S'}\beta(z_v) =\sum_{v\in S\setminus S'}\beta(x'_v)+\sum_{i=1}^N \sum_{v\in \Lambda_i\setminus \{v'_i\}}\beta(x'_v)+ \sum_{v\in  S'}\beta(z_v) \\
	&=-\sum_{v\in  S'}\beta(x'_v)+\sum_{i\in I} \partial_{\beta,D_i} (\Frob_{\zeta_{w,i}})+\sum_{i=n+1}^{N}\partial_{\beta,D_i}(\sigma_i)-\sum_{i\in I}\partial_{\beta,D_i} (\gamma\circ \tau_i)\\
	&=\sum_{i\in I}\partial_{\beta,D_i}(\tau_i)+\sum_{i\in I} \partial_{\beta,D_i} (\Frob_{\zeta_{w,i}})+\sum_{i=n+1}^{N}\partial_{\beta,D_i}(\sigma_i)-\sum_{i\in I}\partial_{\beta,D_i} (\gamma\circ \tau_i)\\
	&= \sum_{i\in I}\partial_{\beta,D_i}(\sigma_i)+\sum_{i=n+1}^{N}\partial_{\beta,D_i}(\sigma_i)\\
	&=\sum_{i=1}^{N}\partial_{\beta,D_i}(\sigma_i).
	\end{align*}
	by (\ref{sigmai}),(\ref{twist-value}) and (\ref{br-value}).
\end{proof}

For any $i \in I$, we replace $x'_{\frak p_i}$ by $z_{\frak p_i}$, we also denote it by $x'_{\frak p_i}$.	
	For each $i\in \{1,\dots,N\}$, by Lemma~\ref{lem:isinhi2} and \ref{br-tr}, we can choose 
	a rational point~$\zeta_{w'_i,i}$ of the fiber of $\sD_i^0\to \mtilde_i$ above~$w'_i$
	such that $\Frob_{\zeta_{w'_i,i}}=-\sigma_i$. Therefore
	\begin{align*}
	\sum_{v \in \Lambda_i} n_{w,i} \mkern2mu\Frob_{\zeta_{w,i}}=0
	\end{align*}
	in~$G_i$.
By Hensel's Lemma, we can lift~$\zeta_{w'_i,i}$ to a $k_{v'_i}$-point~$x'_{v'_i}$ of~$X_{t_0}$. By a similar argument as in the last paragraph of the proof in Proposition \ref{main:1}, we conclude that $\sum_{v \in \Omega}\inv_v \beta(x'_v)=0$ for all $\beta \in B$. \qedhere
	
\end{proof}

\begin{theorem} \label{thm:cyclic}  Let~$X$ be a smooth, proper and geometrically integral variety over  $k$ and let
	$f:X \to \mathbb P^1_k$
	be a dominant morphism with geometrically integral generic fiber $X_{\bar\eta}$.  Let $P_i(t), k_i,L_i$ be as in Theorem \ref{maintheorem}. 
	
	Assume that
	\begin{enumerate}	
	\item $H^1(X_{\bar\eta},\Bbb Q/\Bbb Z)=0$ and $H^2(X_{\bar\eta},\mathscr O_{X_{\bar\eta}})=0$ (e.g., $X_{\bar \eta}$ is rationally connected);	
		
	\item for $i\in \{1,\cdots, n\}$, $L_i$ is a cyclic extension of $k_i$;
	
	\item the hypothesis $({\rm H_1})$  holds for $P_1(t),\cdots, P_n(t)$;
	
	\item there exists a Hilbert subset $H \subset \Bbb P^1_k$ such that
	$X_c(k)$ is dense in $X_c(\mathbf A_k)^{\Br(X_c)}$
	for every rational point~$c$ of~$H$.

	\end{enumerate}	
 Then $X(k)$ is dense in~$X(\mathbf A_k)^{\Br(X)}$.
\end{theorem}
\begin{proof} The condition (1) implies $\Br(X_\eta)/f_\eta^*\mkern.5mu\Br(\eta)$ is finite (see \cite[Lemma 8.6]{HW}). By Proposition \ref{sc-bm}, the proof follows from a similar argument as in Theorem \ref{cor:ratpointsRC}.
\end{proof}

\begin{remark}
	\begin{enumerate} [(1)]
		\item Suppose $k=\Bbb Q$ and all non-split fibers lie over rational points of $\Bbb P^1_\Bbb Q$. By \cite[Corollary 1.9]{GT10},  Hypothesis $({\rm H_1})$  holds (see \cite[Proposition 1.2]{hsw}), then Theorem \ref{thm:cyclic} covers the main results of \cite{bms} and \cite{hsw}. 
		
		\item Let $X$ be a smooth proper model of the smooth locus of the affine variety over $k$ defined by 
		\begin{equation*}
		\prod_{i=1}^m N_{K/k}({\bf z})= c \prod_{j=1}^n P_j(t)^{e_j},
		\end{equation*}
		where all $K_i/k$ are finite field extensions and  the $P_j(t)$'s are distinct irreducible monic  polynomials. Let $L_j=k[t]/(P_j(t))$. For any $j \in \{1,\cdots,n \}$, suppose there exists some $i$ such that  $K_i L_j/L_j$ is cyclic. Suppose $[K_1:k]\mid \sum e_j[L_j:k]$. Then conditions (1), (2) and (4) in Theorem \ref{thm:cyclic} hold. Generally, the Brauer group of the generic fiber is non-trivial, hence Theorem \ref{thm:cyclic} gives many new examples different from \cite[Theorem 1.2]{ctsksd98}.
		\end{enumerate}
\end{remark}

By Theorem \ref{thm:cyclic} and Remark \ref{HB}, we immediately have the following unconditional result.
\begin{cor}
	\label{cor:smallrank3} Let $k=\mathbb Q$. 
	Let~$X$ be a smooth, proper, geometrically integral  variety over
	a number field~$\mathbb Q$
	and $f:X \to \mathbb P^1_\mathbb Q$
	be a dominant morphism
	with rationally connected geometric generic fiber.
	Assume that
	$\mathrm{rank}(f)\leq 3$. Suppose every non-split fiber $X_m$ contains a multiplicity $1$ component $D_m$ such that the algebraic closure of  $k(m)$ in the function field $k(D_m)$ is a cyclic extension of $k(m)$.

	If~$X_c(k)$ is dense in $X_c(\mathbf A_k)^{\Br(X_c)}$ for every
	rational point~$c$ of a Hilbert subset of~$\P^1_k$, then~$X(k)$ is dense in~$X(\mathbf A_k)^{\Br(X)}$.
\end{cor}
\begin{proof} This theorem has been proved (see \cite{HW,bs}) when there is a non-split fiber over a rational point of $\Bbb P^1_\QQ$. In fact, if there is a non-split fiber over a rational point of $\Bbb P^1_\QQ$, this theorem  holds for all $\mathrm{rank}(f)\leq 3$ without the cyclic assumption. Suppose  all non-split fibers do not lie over a rational point of $\Bbb P^1_\QQ$, then $f$ is non-split over only one closed point $m\in \Bbb P^1_\QQ$ with $[\QQ(m):\QQ] =3$. The proof follows from Theorem \ref{thm:cyclic} and Remark \ref{HB}.
\end{proof}


\begin{thebibliography}{CTSanSD87b}


\bibitem[BHB12]{BHB} T. Browning and R. Heath-Brown: \emph{Quadratic polynomials represented by norm forms},  Geom. Funct. Anal. 22 (2012) 1124-1190. 

\bibitem[BM17]{browningmatthiesen}
T.~D. Browning and L.~Matthiesen, \emph{Norm forms for arbitrary number fields
	as products of linear polynomials},
Ann. Sci. Ecole Norm. Sup. \textbf{50} (2017), 1375--1438.

\bibitem[BMS14]{bms}
T.~D. Browning, L.~Matthiesen and A.~N. Skorobogatov, \emph{Rational points on
	pencils of conics and quadrics with many degenerate fibers}, Ann. of Math.
(2) \textbf{180} (2014), no.~1, 381--402.

\bibitem[BS19]{bs}T.~D. Browning and D.~Schindler, \emph{Strong approximation and a conjecture of Harpaz and Wittenberg}, Int. Math. Res. Not.  (2019), no. 14, 4340-4369.









\bibitem[CT03]{ctbudapest}
J.-L. Colliot-Th{\'e}l{\`e}ne, \emph{Points rationnels sur les fibrations}, Higher dimensional
varieties and rational points (Budapest, 2001), Bolyai Soc. Math. Stud.,
vol.~12, Springer, Berlin, 2003, pp.~171\nobreakdash--221.

\bibitem[CT15]{ct15}
\bysame, \emph{Un calcul de groupe de Brauer et une application arithm{\'e}tique},  in Arithmetic and Geometry (Proceedings of the Hausdorff trimester, Bonn 2013). Dieulefait, L.; Faltings, G.; Heath-Brown, R.; Manin, Y.; Moroz, B.; Wintenberger, J.-P. (eds.),  London Math. Soc. Lecture Note Series 420, Cambridge University Press 2015. 



\bibitem[CTH16]{CTH16}
J.-L. Colliot-Th{\'e}l{\`e}ne and D.Harari, \emph{Approximation forte en famille}, J. Reine Angew. Math. \textbf {710} (2016),  173–-198.






\bibitem[CTSa89]{CTSa} 
J.L. Colliot-Th\'el\`ene and P. Salberger: \emph{Arithmetic on some singular cubic hypersurfaces} Proc. Lond. Math. Soc. 58 (1989) 519-549.

\bibitem[CTS82]{ctsansucschinzel}
J.-L. Colliot-Th{\'e}l{\`e}ne and J.-J. Sansuc,, \emph{Sur le principe de {H}asse et l'approximation faible, et sur une hypoth\`ese de {S}chinzel}, Acta Arith. \textbf{41} (1982), no.~1, 33--53.

\bibitem[CTS87]{ctsandescent2}
\bysame, \emph{La descente sur les vari\'et\'es rationnelles. {II}}, Duke Math.
J. \textbf{54} (1987), no.~2, 375--492.

\bibitem[CTSS87]{CTSSD} 
J.L. Colliot-Th\'el\`ene, J.L. Sansuc and P. Swinnerton-Dyer: \emph{Intersections of two quadrics and Ch\^{a}telet surfaces I, II} J. f\"ur die reine und angew. Mathematik  373 (1987) 37--107, 374 (1987) 72-168.


\bibitem[CTS00]{ctskodescent}
J.-L. Colliot-Th{\'e}l{\`e}ne and A.~N. Skorobogatov, \emph{Descent on
	fibrations over {${\bf P}^1_k$} revisited}, Math. Proc. Cambridge Philos.
Soc. \textbf{128} (2000), no.~3, 383--393.


\bibitem[CTSD94]{ctsd94}
J.-L. Colliot-Th{\'e}l{\`e}ne and Sir~Peter Swinnerton-Dyer, \emph{Hasse
	principle and weak approximation for pencils of {S}everi-{B}rauer and similar
	varieties}, J.~reine angew. Math. \textbf{453} (1994), 49--112.



\bibitem[CTSS98]{ctsksd98}
J.-L. Colliot-Th{\'e}l{\`e}ne, A.~N. Skorobogatov, and Sir~Peter
Swinnerton-Dyer, \emph{Rational points and zero-cycles on fibred varieties:
	{S}chinzel's hypothesis and {S}alberger's device}, J.~reine angew. Math.
\textbf{495} (1998), 1--28.



\bibitem[CDX19]{cdx} Y. Cao, C. Demarche, F. Xu: \emph{Comparing descent obstruction and Brauer--Manin obstruction for open varieties},  Trans. Amer. Math. Soc. \textbf{371} (2019), 8625-8650.

\bibitem[CWX19]{CWX19} Y. Cao, D.Wei, F. Xu: \emph{Strong approximation for a family of norm varieties},  preprint, 2019.




\bibitem[CX14]{caoxu}
Y.~Cao and F.~Xu, \emph{Strong approximation with {B}rauer--{M}anin obstruction
	for toric varieties}, Ann. Inst. Fourier (Grenoble), to appear, 2014.


\bibitem[DSW14]{derenthalsmeetswei}
U.~Derenthal, A.~Smeets, and D.~Wei, \emph{Universal torsors and values of
	quadratic polynomials represented by norms}, Math. Ann. (2014), to appear.

\bibitem[Eke90]{ekedahl}
T.~Ekedahl, \emph{An effective version of {H}ilbert's irreducibility theorem},
S\'eminaire de Th\'eorie des Nombres, Paris 1988--1989, Progr. Math.,
vol.~91, Birkh\"auser Boston, Boston, MA, 1990, pp.~241--249.

\bibitem[GHS03]{ghs}
T.~Graber, J.~Harris, and J.~Starr, \emph{Families of rationally connected
	varieties}, J. Amer. Math. Soc. \textbf{16} (2003), no.~1, 57--67.

\bibitem[GT10]{GT10}
B.~Green and T.~Tao,  \emph{Linear equations in primes}, Ann. Math. \textbf{171} (2010) 1753--1850.


\bibitem[Har94]{harariduke}
D.~Harari, \emph{M\'ethode des fibrations et obstruction de {M}anin}, {Duke
	Math. J.} \textbf{75} (1994), no.~1, 221--260.

\bibitem[Har97]{hararifleches}
\bysame, \emph{Fl\`eches de sp\'ecialisations en cohomologie \'etale et
	applications arithm\'etiques}, Bull. Soc. Math. France \textbf{125} (1997),
no.~2, 143--166.

\bibitem[HBM04]{hbm}
D.~R. Heath-Brown and B.~Z. Moroz, \emph{On the representation of primes by
	cubic polynomials in two variables}, Proc. London Math. Soc. (3) \textbf{88}
(2004), no.~2, 289--312.

\bibitem[HBS02]{heathbrownskorobogatov}
R.~Heath-Brown and A.~Skorobogatov, \emph{Rational solutions of certain
	equations involving norms}, Acta Math. \textbf{189} (2002), no.~2, 161--177.


\bibitem[HSW14]{hsw}
Y.~Harpaz, A.~N. Skorobogatov and O.~Wittenberg, \emph{The
	{H}ardy--{L}ittlewood conjecture and rational points}, Compos. Math.
\textbf{150} (2014), no.~12, 2095--2111.


\bibitem[HW16]{HW}
 Y. Harpaz and O. Wittenberg: \emph{On the fibration method for zero-cycles and rational points}, Annals of Mathematics {\textbf 183} (2016), no. 1, 229-295.

\bibitem[Irv14]{irving}
A.~J. Irving, \emph{Cubic polynomials represented by norm forms}, J. reine angew. Math. \textbf{723} (2017) 217-250.



\bibitem[Jo13]{Jo13}
 M. S. Jones: \emph{A note on a theorem of Heath-Brown and Skorobogatov}, Quat. J. Math. {\textbf 64} (2013) 1239-1251.

\bibitem[LW54]{langweil}
S.~Lang and A.~Weil, \emph{Number of points of varieties in finite fields},
Amer. J. Math. \textbf{76} (1954), 819\nobreakdash--827.


\bibitem[Man71]{Ma71}
Yu.~I. Manin, \emph{Le groupe de {B}rauer-{G}rothendieck en g\'eom\'etrie
	diophantienne}, Actes du {C}ongr\`es {I}nternational des {M}ath\'ematiciens
({N}ice, 1970), {T}ome 1, Gauthier-Villars, Paris, 1971, pp.~401--411.


\bibitem[Man86]{Ma86}
Yu.~I. Manin, \emph{Cubic forms: algebraic, geometry and arithmetic}, North Holland,Second edition, 1986.


\bibitem[Mat18]{matthiesen}
L.~Matthiesen, \emph{On the square-free representation function of a norm form
	and nilsequences}, J. Inst. Math. Jussieu \textbf{17} (2018), 107--135.




%




\bibitem[SD99]{SD99}  
P.~Swinnerton-Dyer, \emph{Rational points on pencils of conics with 6 singular fibres}, Ann. Fac. Sci. Toulouse {\bf 8} (1999) 331-341. 


 

\bibitem[SD11]{swdtopics}
\bysame, \emph{Topics in {D}iophantine equations}, Arithmetic geometry, Lecture
Notes in Math., vol. 2009, Springer, Berlin, 2011, pp.~45--110.
%
%

\bibitem[SJ13]{swarbrickjones}
M.~Swarbrick~Jones, \emph{A note on a theorem of {H}eath-{B}rown and
	{S}korobogatov}, Q. J. Math. \textbf{64} (2013), no.~4, 1239--1251.


\bibitem[Sko90]{skofibration}
A.~N. Skorobogatov, \emph{On the fibration method for proving the {H}asse principle and
	weak approximation}, S\'eminaire de {T}h\'eorie des {N}ombres, {P}aris
1988--1989, Progr. Math., vol.~91, Birkh\"auser Boston, Boston, MA, 1990,
pp.~205--219.

\bibitem[Sko96]{skorodescent}
\bysame, \emph{Descent on fibrations over the projective line}, Amer. J. Math.
\textbf{118} (1996), no.~5, 905--923.




\bibitem[Sko01]{skobook}
\bysame, \emph{Torsors and rational points}, Cambridge Tracts in Mathematics,
vol. 144, Cambridge University Press, Cambridge, 2001.



\bibitem[Wei14a]{weioneq}
D.~Wei, \emph{On the equation {$N_{K/k}(\Xi)=P(t)$}}, Proc. Lond. Math. Soc.
(3) \textbf{109} (2014), no.~6, 1402--1434.

\bibitem[Wei14b]{weitorus}
\bysame, \emph{Strong approximation for a toric variety},
Acta Math. Sinca, to appear, 2014.

\bibitem[Wei16]{weidescent}
\bysame, \emph {Open descent and strong approximation}, \href{http://arxiv.org/abs/1604.00610}{arXiv:1604.00610}, 2016.

\bibitem[Wit07]{wittlnm}
O.~Wittenberg, \emph{Intersections de deux quadriques et pinceaux de courbes de
	genre~$1$}, Lecture Notes in Mathematics, vol. 1901, Springer, Berlin, 2007.

\bibitem[Wit18]{wit18}
\bysame,
\emph{Rational points and zero-cycles on rationally connected varieties over number fields},
in Algebraic Geometry: Salt Lake City 2015, Part 2, Proceedings of Symposia in Pure Mathematics 97, American Mathematical Society, Providence, RI, 2018,  597--635.


\end{thebibliography}
\end{document}